\documentclass[a4paper,11pt]{amsart}
\usepackage{amsfonts}
 \usepackage{enumerate}
 \usepackage{amsmath, amsthm, amscd, amssymb}
\usepackage{pdfsync}
\usepackage[title]{appendix}
\usepackage[all]{xy}
\usepackage{latexsym,graphicx}
\usepackage[latin1]{inputenc}
\usepackage{xspace}
\usepackage{array}
\usepackage{newclude}
\usepackage{hyperref}
\usepackage{url}
\usepackage{enumitem}
\hypersetup{
  colorlinks,
  citecolor=green,
  linkcolor=magenta,
  urlcolor=black}
  
\renewcommand*{\HyperDestNameFilter}[1]{\jobname-#1} 
\numberwithin{equation}{section}
\usepackage[centering, marginparwidth=2cm]{geometry}
\usepackage{marginnote}

\addtocontents{toc}{\setcounter{tocdepth}{1}}

\usepackage[nameinlink]{cleveref}

\usepackage{tikz-cd}
\usetikzlibrary{cd}
\usepackage{leftidx}
\usetikzlibrary{positioning}
\usepackage{dsfont}
\usepackage{tikz}
\usetikzlibrary{matrix,arrows,decorations.pathmorphing,positioning}

\newcommand\blfootnote[1]{%
\begingroup
\renewcommand\thefootnote{}\footnote{#1}%
\addtocounter{footnote}{-1}%
\endgroup
}
  
  \newcommand{\Addresses}{{
  \bigskip
  \footnotesize

\textsc{I.H.E.S., Universit\'e Paris-Saclay, CNRS, Laboratoire Alexandre
  Grothendieck. 35 Route de Chartres, 91440 Bures-sur-Yvette
  (France)}\par\nopagebreak
  \textit{E-mail address}, G.~Baldi: \texttt{baldi@ihes.fr}  

  \medskip

  \textsc{Humboldt Universit\"{a}t zu Berlin (Germany)}
  \par\nopagebreak
  \textit{E-mail address}, B.~Klingler: \texttt{bruno.klingler@hu-berlin.de}
  
  \medskip
  
\textsc{I.H.E.S., Universit\'e Paris-Saclay, CNRS, Laboratoire Alexandre Grothendieck. 35 Route de Chartres, 91440 Bures-sur-Yvette (France)}\par\nopagebreak
  \textit{E-mail address}, E.~Ullmo: \texttt{ullmo@ihes.fr}
}}

\theoremstyle{plain}
\newtheorem{theor}{Theorem}[section]

\newtheorem{conj}[theor]{Conjecture}
\newtheorem{lem}[theor]{Lemma}
\newtheorem{defi}[theor]{Definition}
\newtheorem{nota}[theor]{Notation}

\newtheorem{prop}[theor]{Proposition}
\newtheorem{cor}[theor]{Corollary}
\theoremstyle{definition}

\newtheorem{example}[theor]{Example}

\newtheorem{rmk}[theor]{Remark}
\newtheorem{question}[theor]{Question}
\theoremstyle{remark}
\numberwithin{equation}{subsection}

\newcommand{\DT}{\mathbf{S}}

\newcommand{\RR}{\mathbb{R}}
\newcommand{\CC}{\mathbb{C}}
\newcommand{\NN}{\mathbb{N}}
\newcommand{\ZZ}{\mathbb{Z}}
\newcommand{\QQ}{\mathbb{Q}}

\newcommand{\HH}{{\mathbf H}}
\newcommand{\LL}{{\mathbf L}}

\newcommand{\ws}{\textnormal{ws}}

\newcommand{\cE}{\mathcal{E}}

\newcommand{\noi}{\noindent}

\newcommand{\alg}{\operatorname{alg}}
\newcommand{\an}{\textnormal{an}}

\newcommand{\Hcd }{\operatorname{H-cd}}
\newcommand{\hHcd }{\operatorname{hH-cd}}
\newcommand{\Mcd }{\operatorname{M-cd}}
\newcommand{\hMcd }{\operatorname{hM-cd}}
\newcommand{\PU}{\operatorname{PU}}
\newcommand{\SU}{\operatorname{SU}}
\newcommand{\SO}{\operatorname{SO}}
\newcommand{\MM}{\mathbf{M}}

\newcommand{\ad}{\textnormal{ad}}
\newcommand{\der}{\textnormal{der}}

\newcommand{\MT}{\mathbf{MT}}

\newcommand{\Ad}{\operatorname{Ad}}

\newcommand{\sing}{\textnormal{sing}}

\DeclareMathOperator{\Sh}{Sh}

\DeclareMathOperator{\Hom}{Hom}

\DeclareMathOperator{\Res}{Res}

\DeclareMathOperator{\Sl}{SL}

\newcommand{\Zar}{\textnormal{Zar}}

\newcommand{\Gm}{\mathbb{G}_{\textnormal{m}}}

\newcommand{\PP}{\mathbf{P}}

\newcommand{\codim}{\operatorname{codim}}

\newcommand{\VV}{{\mathbb V}}

\newcommand{\HL}{\textnormal{HL}}
\newcommand{\atyp}{\textnormal{atyp}}
\newcommand{\watyp}{\textnormal{w-atyp}}
\newcommand{\typ}{\textnormal{typ}}
\newcommand{\pos}{\textnormal{pos}}
\newcommand{\fpos}{\textnormal{f-pos}}

\newcommand{\TT}{\mathbf{T}}

\newcommand{\nc}{\operatorname{nc}}

\newcommand{\N}{\mathbb{N}}

\newcommand{\prim}{\textnormal{prim}}

\newcommand{\Z}{\mathbb{Z}}
\newcommand{\Q}{\mathbb{Q}}

\newcommand{\R}{\mathbb{R}}

\newcommand{\F}{\mathcal{F}}
\newcommand{\bS}{\mathbf{S}}

\usepackage[normalem]{ulem}

\newcommand{\Hh}{\mathbb{H}}

\newcommand{\C}{\mathbb{C}}

\newcommand{\Qbar}{\overline{\mathbb{Q}}}

\newcommand{\GL}{\mathbf{GL}}
\newcommand{\SL}{\mathbf{SL}}

\newcommand{\G}{{\mathbf G}}

\newcommand{\fg}{\mathfrak{g}}
\newcommand{\fh}{\mathfrak{h}}
\newcommand{\fl}{\mathfrak{l}}
\newcommand{\ft}{\mathfrak{t}}
\newcommand{\fm}{\mathfrak{m}}
\newcommand{\fs}{\mathfrak{s}}
\newcommand{\fq}{\mathfrak{q}}
\newcommand{\fr}{\mathfrak{r}}
\newcommand{\fn}{\mathfrak{n}}

\newcommand{\fu}{\mathfrak{u}}

\newcommand{\cA}{\mathcal{A}}

\newcommand{\cO}{\mathcal{O}}

\newcommand{\cV}{{\mathcal V}}

\newcommand{\simple}{\textnormal{simple}}

\usepackage{microtype}

\begin{document}

\title{On the distribution of the Hodge locus}\date{\today}  
\author{Gregorio Baldi, Bruno Klingler, and Emmanuel Ullmo} 

\begin{abstract} Given a polarizable $\ZZ$-variation of Hodge structures $\mathbb{V}$
  over a complex smooth quasi-projective
  base $S$, a classical result of Cattani, Deligne and Kaplan says that its
  Hodge locus (i.e. the locus where exceptional Hodge
  tensors appear) is a {\em countable} union of irreducible algebraic
  subvarieties of $S$, called the special subvarieties for
  $\mathbb{V}$. Our main result in this paper is that, if the level of
  $\VV$ is at least $3$, this Hodge locus is in fact a {\em finite} union of
  such special subvarieties (hence is algebraic), at least if we
  restrict ourselves to the Hodge locus factorwise of positive
  period dimension (\Cref{corol}). For instance the Hodge locus of
  positive period dimension of the universal family of degree $d$
    smooth hypersurfaces in $\mathbf{P}^{n+1}_\mathbb{C}$, 
     $n\geq 3, d\geq 5$ and $(n,d)\neq (4,5)$, is algebraic. On the other hand we prove that in
    level $1$ or $2$, the Hodge locus is analytically dense in $S^\an$
    as soon as it contains one typical special subvariety. These
    results follow from a complete
   elucidation of the distribution in $S$ of
  the special subvarieties in terms of typical/atypical intersections, with the exception of the atypical special
  subvarieties of zero period dimension. 
\end{abstract}
\maketitle

\tableofcontents

\section{Motivation and first main result}

\subsection{Hodge locus and special subvarieties}\label{section0}
Let $f: X \to S$ be a smooth projective morphism of smooth irreducible complex
quasi-projective varieties. Motivated by the study of the Hodge conjecture for the
fibres of $f$, one defines the Hodge locus $\HL(S, f)$ as the locus of
points $s\in S^{\an}= S(\CC)$ for which the Hodge structure
$H^\bullet(X_s^{\an},\ZZ)_\prim/(\textnormal{torsion})$ admits more \emph{Hodge
tensors} than the primitive cohomology of the very
general fibre. Here a Hodge class of a pure $\ZZ$-Hodge structure $V=(V_\ZZ, F^\bullet)$ is a class in $V_\QQ$
whose image in $V_\CC$ lies in the zeroth piece $F^0 V_\CC$ of the
Hodge filtration, or equivalently a morphism of Hodge 
structures $\QQ(0) \to V_\QQ$; and a Hodge tensor for $V$ is a Hodge class
in $V^\otimes:= \bigoplus_{a,b\in
  \N} V^{\otimes a} \otimes( V^\vee)^{\otimes b}$, where
$V^\vee$ denotes the Hodge structure dual to $V$. In this geometric case Weil \cite{Weil}
asked whether $\HL(S, f)$ is a countable union of closed
algebraic subvarieties of $S$ (he noticed that a positive
answer follows easily from the rational Hodge conjecture).

More generally, let
$\VV:=(\VV_\ZZ, \cV, F^\bullet, \nabla)$ be a polarizable variation of
$\ZZ$-Hodge 
structures ($\ZZ$VHS) on $S$. Thus $\VV_\ZZ$ is a finite rank locally free $\ZZ_{S^{\an}}$-local system on
the complex manifold $S^{\an}$ and $(\cV, F^\bullet,
\nabla)$ is a filtered regular algebraic flat connection on
$S$ such that the analytification $(\cV^{\an}, \nabla^{\an})$ is isomorphic
to $\VV_\ZZ \otimes_{\ZZ_{S^{\an}}} \cO_{S^{{\an}}}$ endowed
with the holomorphic flat connection defined by $\VV_\ZZ$ (the
filtration $F^\bullet$ is called the Hodge filtration). The Hodge locus $\HL(S, \VV^\otimes)$ is the subset of points $s \in 
S^{\an}$ for which the Hodge structure $\VV_s$ admits more Hodge tensors
than the very general fiber $\VV_{s'}$.
We recover the geometric
situation by considering the polarizable $\ZZ$VHS ``of geometric
origin'' 
\begin{displaymath}
\VV = \left( \bigoplus_{k \in \NN} ( R^{k}f^{\an}_* \ZZ)_{\operatorname{prim}}/ (\textnormal{torsion}), \cV= 
\bigoplus_{k \in \NN} R^{k}f_* \Omega^\bullet_{X/S}, F^\bullet,
\nabla\right).
\end{displaymath}
In this case the
Hodge filtration $F^\bullet$ on $\cV$ is
induced by the stupid filtration on the algebraic de Rham complex
$\Omega^\bullet_{X/S}$ and $\nabla$ is the Gau\ss\--Manin
connection.

Cattani, Deligne and Kaplan \cite[Theorem 1.1]{CDK95} proved the following
celebrated result, which in particular answers positively Weil's
question (we also refer to \cite[Theorem 1.6]{BKT} for an alternative
proof, using o-minimality):

\begin{theor}[Cattani-Deligne-Kaplan] \label{CDK}
Let $S$ be a smooth connected complex quasi-projective algebraic variety and
$\VV$ be a polarizable $\ZZ$VHS over $S$. 
Then $\HL(S, \VV^\otimes)$ is a countable union of
closed irreducible algebraic subvarieties of $S$: the strict special
subvarieties of $S$ for $\VV$.
\end{theor}
We refer to \Cref{defsp} for a detailed description of the \emph{special} subvarieties of $S$. 
Accordingly to such a definition, $S$ itself is a special subvariety. We are therefore concerned with the 
distribution of the strict special subvarieties of $S$ (i.e. the ones different from $S$).

\subsection{Distribution of special subvarieties}
After \Cref{CDK} one would like to 
understand the distribution in $S$ of the (strict) special
subvarieties for $\VV$. For instance (see \cite[Question 1.2]{klin}): are there any geometric
constraints on the Zariski closure of $\HL(S, \VV^\otimes)$?

The fundamental tool for studying these questions is the
  holomorphic period
  map 
\begin{equation} \label{period0}
  \Phi: S^{\an} \to \Gamma \backslash D,
\end{equation}
completely describing the $\ZZ$VHS $\VV$ (see \Cref{preliminaries} for more details on what follows).
Here $(\G, D)$ denotes the generic Hodge datum of $\VV$
and 
$\Gamma \backslash D$ is the associated Hodge variety. The
Mumford-Tate domain $D$ decomposes as a product $D_1 \times \cdots
\times D_k$, according to the decomposition of the adjoint group $\G^\ad$ into a product $\G_1
\times \cdots \times \G_k$ of simple factors (notice that some factors
$\G_i$ may be $\RR$-anisotropic). Replacing $S$ by a finite \'etale
cover and reordering the factors if necessary, the lattice $\Gamma \subset
\G^\ad(\R)^+$ decomposes as a direct product $\Gamma
\cong \Gamma_1 \times \cdots \times \Gamma_r$, $ r \leq k$, where $\Gamma_i\subset
\G_i(\R)^+$, $1\leq i\leq r$, is an arithmetic lattice, Zariski-dense in
$\G_i$. Writing $D' = D_{r+1} \times \cdots \times D_k$ for the
product of factors where the monodromy is trivial (it contains in
particular all the factors $D_i$ for which $\G_i$ is
$\RR$-anisotropic), the period map is written
\begin{equation} \label{period}
  \Phi: S^{\an} \to \Gamma \backslash D \cong \Gamma_1 \backslash D_1
  \times \cdots \times \Gamma_r
  \backslash D_r \times D'\;\;,
\end{equation}
and the projection of $\Phi(S^\an)$ on $D'$ is a point. For more details, see also \emph{the structure theorem for VHS} from \cite[(III.A.2)]{MR2918237}.

By enlarging $S$ if necessary, one can assume
without loss of generality that $\Phi$ is proper (see \cite[Theorem 9.5]{Griffiths}). The image $\Phi(Z^\an)$ for
any closed algebraic subvariety $Z\subset S$ is then a closed analytic
subvariety of $\Gamma \backslash D$, by Remmert's proper mapping
theorem. It is thus natural to distinguish between special subvarieties $Z$
of \emph{zero period dimension} (i.e. $\Phi(Z^\an)$ is a point of $\Gamma
\backslash D$), which are geometrically elusive; and those of \emph{positive
period dimension} (i.e. $\dim_\CC \Phi(Z^\an) >0$), that are susceptible
of a variational study. Taking into account product
phenomena, we are lead to:

\begin{defi}  \label{positive period dimension} \label{fpositive} \hfill
  \begin{enumerate}
    \item
  A subvariety $Z$ of $S$ is said of {\em positive period dimension for $\VV$} if $\Phi(Z^{\an})$ has
  positive dimension. If moreover the projection of $\Phi(Z^{\an})$ on each factor $\Gamma_i \backslash
  D_i$, $1\leq i \leq r$, has positive dimension, then $Z$ is said to
  be \emph{factorwise of positive period dimension}.
  
 \item
  The {\em Hodge locus of positive period dimension} (resp. {\em factorwise of
  positive period dimension}) 
  $\HL(S, \VV^\otimes)_{\pos}$ (resp. $\HL(S, \VV^\otimes)_{\fpos})$ is the
  union of the special subvarieties of $S$
  for $\VV$ of positive period dimension (resp. factorwise of positive
  period dimension). 
\end{enumerate}
\end{defi}

Thus $\HL(S, \VV^\otimes)_{\fpos} \subset \HL(S, \VV^\otimes)_{\pos}
\subset \HL(S, \VV^\otimes)$ and the first inclusion is an equality if
$\G^\ad$ is simple.

\medskip
Recently Otwinowska and the second author \cite{KO} proved the
following theorem (they work in the case where $\G^\ad$ is simple,
but their proof adapts immediately to the general case):

\begin{theor}[{\cite[Theorem 1.5]{KO}}] \label{KO}
Let $\VV$ be a polarizable $\ZZ$VHS on a smooth connected complex quasi-projective variety
$S$. Then either $\HL(S, \VV^{\otimes})_\fpos$ is an algebraic
subvariety of $S$ (not necessarily irreducible) or it is Zariski-dense in $S$.
\end{theor}

\begin{rmk}
Saying that $\HL(S, \VV^{\otimes})_\fpos$ is algebraic is equivalent to
saying that the set of strict special subvarieties of $S$ for 
$\VV$ factorwise of positive period dimension has only finitely many maximal
elements for the inclusion.
\end{rmk}

\subsection{An algebraicity result}
\Cref{KO} has an important flaw: it does not provide any criterion for
deciding which branch of the alternative holds true. 
Our most striking result in this paper 
proves that $\HL(S, \VV^{\otimes})_\fpos$ is algebraic in most cases.
A simple measure of the complexity of $\VV$ is its level:
roughly, the length of the Hodge filtration on the holomorphic
  tangent space of $D$. See \Cref{levelVHS} for the precise definition in terms of the algebraic
  monodromy group of $\VV$.  While special subvarieties usually abound for  
$\ZZ$VHSs of level one (e.g. families of abelian varieties or families
of K3 surfaces) and for some $\ZZ$VHS of level two (e.g Green's famous example
of the Noether-Lefschetz locus for degree $d$ ($d>3$) surfaces in
$\PP^3_\CC$, see \cite[Proposition 5.20]{Voisin}), we show:

\begin{theor} \label{corol}
 Let $\VV$ be a polarizable $\ZZ$VHS
on a smooth connected complex quasi-projective variety 
$S$. If $\VV$ is of level at least $3$ then $\HL(S,
\VV^\otimes)_\fpos$ is a finite union of maximal atypical special
subvarieties (hence is algebraic). In particular, if moreover $\G^{\ad}$ is
simple, then $\HL(S, \VV^\otimes)_\pos$
 is algebraic in $S$.
\end{theor}

As a simple geometric illustration of \Cref{corol}:
\begin{cor} \label{hypersurface}
Let $\PP^{N(n, d)}_\CC$ be the projective space parametrising the hypersurfaces $X$ of
$\PP^{n+1}_\CC$ of degree $d$ (where $N(n, d)=\binom{n+d+1}{d}-1$). Let $U_{n, d} \subset \PP^{N(n, d)}_\CC$ be
the Zariski-open subset parametrising the smooth
hypersurfaces $X$ and let 
$
\VV \to U_{n,d}
$
be the $\Z$VHS corresponding to the primitive cohomology
$H^n(X, \ZZ)_\prim$. 

If $n=3$ and $d \geq 5$; $n=4$ and $d\geq 6$; $n=5,6,8$ and $d \geq 4$; and $n=7$ or $\geq 9$ and $d \geq 3$, then the level of $\VV\to U_{n,d}$ is at least 3, and therefore $\HL(U_{n,d}, \VV^\otimes)_\pos \subset U_{n,d}$ is algebraic.
\end{cor}

\noi
Let us also mention that \Cref{corol}, combined with the main result of
\cite{LV}, provides a nice first result in the direction of the
so called \emph{refined Bombieri-Lang
conjecture}. See our note \cite{BKU2}.

\section{Special subvarieties as (a)typical intersection loci:
  conjectures}

The heuristic behind the proof of \Cref{corol} is based on a remarkable feature of Hodge theory: thanks to the
existence of period maps, special subvarieties 
can also be defined as {\em intersection loci}. Indeed, a closed irreducible subvariety $Z \subset S$ is special
for $\VV$ (we will equivalently say that it is special for $\Phi$)
precisely if $Z^{\an}$ coincides with an analytic irreducible 
component $\Phi^{-1}(\Gamma' \backslash D')^0$ of $\Phi^{-1}(\Gamma' \backslash D')$, for $(\G', D') \subset
(\G, D)$ the generic Hodge subdatum of $Z$ and $\Gamma' \backslash D'
\subset \Gamma \backslash D$ the associated Hodge subvariety.

\begin{rmk}
The description of special subvarieties as intersection loci makes the study
of the Hodge locus simpler than the study of the Tate locus (its
analogue obtained when replacing $\VV$ by a lisse $\ell$-adic sheaf over a smooth variety $S$ over a
field of finite type and the Hodge tensors by the Tate tensors).
\end{rmk}

This suggests a fundamental dichotomy between \emph{typical} and \emph{atypical} 
intersections, which should govern whether or not the Hodge locus
$\HL(S, \VV^\otimes)$ is algebraic, and in particular which branch of the alternative is
satisfied in \Cref{KO}. Such a dichotomy was 
proposed for the first time by the second author in \cite{klin}:

\begin{defi}\label{atypical}
  Let $Z = \Phi^{-1}(\Gamma' \backslash D')^0\subset S$ be a special
  subvariety for $\VV$, with generic Hodge datum $(\G', D')$. It is
  said to be \emph{atypical} if $\Phi(S^{\an})$ and $\Gamma'
  \backslash D'$ do not intersect generically along $\Phi(Z)$: 
  \begin{equation} \label{equation atypical}
    \codim_{\Gamma\backslash D} \Phi(Z^{\an}) < \codim_{\Gamma\backslash
    D} \Phi(S^{\an}) + \codim_{\Gamma\backslash D} \Gamma'\backslash
  D'\;\;,
  \end{equation}
  or if $Z$ is singular for $\VV$
  (meaning that $\Phi(Z^\an)$ is contained in the singular locus of $\Phi(S^\an)$).
  \noindent
  Otherwise it is said to be \emph{typical}.
\end{defi}

\begin{rmk} \label{singular}
  In \Cref{atypical}, deciding that if $Z$ is singular for $\VV$ then
  it is atypical for $\VV$ hides the fact that the numerical
  condition~(\ref{equation atypical}) is too naive when $Z$ is
  singular. This is the right convention for hoping
  \Cref{conj-typical} below to be true. Notice that if we define the {\em
    singular locus of $S$ for $\VV$} as the preimage $S^\sing_\VV$
  under $\Phi$ of the singular 
locus of the complex analytic variety $\Phi(S^\an)$ (in particular any
$Z\subset S$ special and singular for $\VV$ is contained in
$S^\sing_\VV$), it follows from the
  definability of $\Phi$ in the o-minimal structure $\RR_{\an, \exp}$
  \cite[Theorem 1.3]{BKT} that $S^\sing_\VV$ is actually a closed (strict)
  algebraic subvariety of $S$ (not necessarily irreducible), see \cite[Section 8]{KO} for
  a proof. The precise structure of the singular special locus is
  described in \Cref{below}.
\end{rmk}

\begin{defi} \label{atypical locus}
  The \emph{atypical Hodge locus} $\HL(S,\VV^\otimes)_{\atyp} \subset \HL(S, \VV^\otimes)$
  (resp. the \emph{typical Hodge locus} $\HL(S,\VV^\otimes)_{\typ} \subset \HL(S, \VV^\otimes)$) is
  the union of the atypical (resp. strict typical) special subvarieties of $S$ for $\VV$.
\end{defi}

\subsection{Conjectures}
We expect the Hodge locus
\begin{displaymath}
\HL(S, \VV^\otimes) = \HL(S,\VV^\otimes)_{\atyp} \, \bigcup\, 
\HL(S,\VV^\otimes)_{\typ}
\end{displaymath}
 to satisfy the
following two conjectures:

\begin{conj}[Zilber--Pink conjecture for the atypical Hodge locus, strong version] \label{main conj}
  Let $\VV$ be a polarizable $\ZZ$VHS on an irreducible smooth
  quasi-projective variety $S$. The atypical Hodge locus
  $\HL(S,\VV^\otimes)_{\atyp}$ is a finite union of atypical special
subvarieties of $S$ for $\VV$. Equivalently: the set of atypical
special subvarieties of $S$ for $\VV$ has finitely many maximal elements for
the inclusion.
\end{conj}

\begin{rmk}
\Cref{main conj} is stronger than the Zilber-Pink conjecture for
$\ZZ$VHS originally proposed in \cite[Conjecture 1.9]{klin}. Indeed,
our \Cref{atypical} of atypical special subvarieties is more general
(and simpler) than the one considered in {\em op. cit.}. We refer to \Cref{newzpsection} for a
detailed comparison with \cite{klin}.
\end{rmk}

\begin{conj}[Density of the typical Hodge locus] \label{conj-typical}
Let $\VV$ be a polarizable $\ZZ$VHS on a smooth connected complex quasi-projective variety
$S$. If 
$\HL(S, \VV^\otimes)_\typ$ is not empty then it is dense (for the analytic topology) in
$S$.
\end{conj}

\Cref{main conj} and \Cref{conj-typical} imply immediately the
following, which clarifies the possible alternatives in \Cref{KO}:
\begin{conj} \label{second-main}
  Let $\VV$ be a polarizable $\ZZ$VHS on an irreducible smooth
  quasi-projective variety $S$. If $\HL(S, \VV^\otimes)_\typ$ is
  empty then $\HL(S, \VV^\otimes)$ is algebraic; otherwise
  $\HL(S, \VV^\otimes)_\typ$ is analytically dense in $S$.
  \end{conj}

In view of \Cref{conj-typical} and \Cref{second-main}, we are led
to the:

\begin{question} \label{question}
  Is there a simple combinatorial criterion on $(\G, D)$ for deciding
  whether $\HL(S, \VV)_\typ$ is empty (which would imply, following
  \Cref{second-main}, that $\HL(S, \VV^\otimes)$ is algebraic)?
\end{question}

\section{Special subvarieties as (a)typical intersection loci: results and applications}
In this paper we prove \Cref{main conj} and \Cref{conj-typical} (and
thus \Cref{second-main}) except for the description of
the atypical locus of zero period dimension, see
\Cref{geometricZP} and \Cref{typicallocus} respectively. We also provide a powerful criterion for $\HL(S,
\VV^\otimes)_\typ$ to be empty, thus answering \Cref{question}, see
\Cref{level criterion}. Our \Cref{corol} that ``in most cases'' $\HL(S,
\VV^\otimes)_\fpos$ is algebraic then follows from \Cref{geometricZP} and
\Cref{level criterion}. Let us notice that all these results are independent of
 \Cref{KO}, which we do not use and of which we prove a variant, see \Cref{KO1}.

\subsection{On the atypical Hodge locus}

Our first main result establishes the \emph{geometric part} of
\Cref{main conj}: we prove that the maximal atypical special subvarieties of positive
period dimension arise in a finite number of families whose geometry
we control. We cannot say anything on the atypical locus of zero
period dimension, for which different ideas are certainly needed.

\begin{theor}[Geometric Zilber--Pink]\label{geometricZP}
Let $\VV$ be a polarizable $\ZZ$VHS on a smooth connected complex quasi-projective variety
$S$. Let $Z$ be an irreducible component of the Zariski closure
of $\HL(S, \VV^{\otimes})_{\pos, \atyp}:= \HL(S,
\VV^\otimes)_{\pos} \cap \HL(S,
\VV^\otimes)_{\atyp}$ in $S$. Then:

\begin{itemize}
  \item[(a)] Either $Z$
    is a maximal atypical special subvariety;
    \item[(b)] 
Or the generic adjoint Hodge datum $(\G_Z^\ad, D_{G_{Z}})$
  decomposes as a non-trivial product $(\G', D') \times (\G'', D'')$,
  inducing (after replacing $S$ by a finite \'{e}tale cover if
  necessary) $$
\Phi_{|Z^\an}= (\Phi', \Phi''): Z^\an \to  \Gamma_{\G_{Z}}\backslash D_{G_{Z}}= \Gamma'
\backslash D' \times  \Gamma'' \backslash D''\subset \Gamma \backslash
D
\;\;,$$
such that $Z$ contains a Zariski-dense
set of atypical special subvarieties for $\Phi''$ of zero
period dimension. Moreover $Z$ is Hodge generic in a special subvariety $$\Phi^{-1}(
\Gamma_{\G_{Z}}\backslash D_{G_{Z}})^0$$ of $S$ for $\Phi$ which is typical. 
 \end{itemize}
\end{theor}

\begin{rmk}
\Cref{main conj}, which also takes into account the atypical special
subvarieties of zero period dimension, predicts that the branch (b) of the
  alternative in the conclusion of \Cref{geometricZP} never occurs.
\end{rmk}

In \Cref{Shimura} and \Cref{modular} we discuss special examples
of \Cref{geometricZP} of particular interest.

\subsection{A criterion for the typical Hodge locus to be empty} \label{criterion}

The following result answers~\Cref{question}:
\begin{theor} \label{level criterion}
Let $\VV$ be a polarizable $\ZZ$VHS
on a smooth connected complex quasi-projective variety 
$S$, with generic Hodge datum $(\G, D)$ and algebraic monodromy group
$\HH$. If $\VV$ is of level at
 least $3$ then $\HL(S,\VV^\otimes)_{\typ} = \emptyset$ (and thus $\HL(S,
 \VV^\otimes)= \HL(S, \VV^\otimes)_\atyp$).
\end{theor}

\noindent
In level~$2$ we prove that strict typical special
subvarieties do satisfy some geometric constraints:

\begin{prop} \label{level22}
Let $\VV$ be a polarizable $\ZZ$VHS
on a smooth connected complex quasi-projective variety 
$S$ of level $2$. Suppose that the Lie algebra $\fg^\ad$ of its
adjoint generic Mumford-Tate group is simple. If $Z \subset S$ is a typical special subvariety
then its adjoint generic Mumford-Tate group $\G^{\ad}_Z$ is simple.
\end{prop}

\subsection{On the algebraicity of the Hodge locus} \label{mainsectionintro}
As we will show in \Cref{proof corol}, \Cref{corol} easily follows 
  from \Cref{geometricZP} and \Cref{level criterion}.
Notice that the full \Cref{main conj}, allied with \Cref{level
  criterion}, would imply in the same way:

\begin{conj} \label{algebraicity}
Let $\VV$ be a polarizable $\ZZ$VHS
on a smooth connected complex quasi-projective variety 
$S$. If $\VV$ is of level at least $3$ then $\HL(S, \VV^\otimes)$ is
algebraic.
\end{conj}

\begin{rmk}
  Notice that \Cref{algebraicity} implies, if $S$ is defined over
  $\Qbar$, that $S$ contains a Hodge generic $\overline{\QQ}$-point, as predicted, at least if $\VV$ is of geometric origin (in the sense of \Cref{section0}),
  by the conjecture that Hodge classes are absolute Hodge classes.
\end{rmk}

\medskip
In level two, the typical locus $\HL(S, \VV^\otimes)_\typ$ is not
necessarily empty, but, thanks to \Cref{level22}, we can still prove:

\begin{theor}\label{level2}
Let $\VV$ be a polarizable $\ZZ$VHS
on a smooth connected complex quasi-projective variety 
$S$ with generic Mumford-Tate datum $(\G, D)$. Suppose that $\VV$ has
level $2$ and that $\G^\ad$ is simple. Then $\HL(S,
\VV^\otimes)_{\pos, \atyp}$ is algebraic.
\end{theor}

\begin{cor} \label{hypersurface2}
Let $\PP^{N}_\CC$ be the projective space parametrising the hypersurfaces $X$ of
$\PP^{3}_\CC$ of degree $d$ (with $N=\frac{(d+3)(d+2)(d+1)}{6}-1$). Let $U_{2, d} \subset \PP^{N}_\CC$ be
the Zariski-open subset parametrising the smooth
hypersurfaces $X$ and let 
$
\VV \to U_{2,d}
$
be the $\Z$VHS corresponding to the primitive cohomology
$H^2(X, \ZZ)_\prim$. If $d\geq 5$ then $\HL(U_{2,d},
\VV^\otimes)_{\pos, \atyp} \subset U_{2,d}$ is algebraic.
\end{cor}

\subsection{On the typical Hodge locus in level one and two}\label{typical}
Let us now turn to the typical Hodge locus $\HL(S,\VV^\otimes)_{\typ}
$. 

\subsubsection{The ``all or nothing'' principle}
In view of \Cref{level criterion}, $\HL(S,\VV^\otimes)_{\typ}
$ is non-empty only if $\VV$ is
of level one or two. In that case its behaviour is predicted by
\Cref{conj-typical}. In this direction we obtain our last main result:

\begin{theor}\label{typicallocus}
Let $\VV$ be a polarizable $\ZZ$VHS
on a smooth connected complex quasi-projective variety 
$S$. If the typical Hodge locus $\HL(S,\VV^\otimes)_{\typ} $ is non-empty then
$\HL(S,\VV^\otimes)$ is analytically (hence Zariski) dense in
$S$.
\end{theor}

Here and elsewhere, by analytically dense, we mean dense in
  every non-empty open subset of $S$ (for the usual Euclidean
  topology).

\begin{rmk}
Notice that, in \Cref{typicallocus}, we also treat
the typical Hodge locus of zero period dimension. 
\end{rmk}

\begin{rmk}\label{rmkchai}
\Cref{typicallocus} is new even for $S$ a subvariety of a Shimura variety. Its proof is inspired by the arguments of Colombo-Pirola
\cite{colombopriola}, Izadi \cite{izadi} and Chai
\cite{chai} in that case.
\end{rmk}

\noi
In the same way that \Cref{main conj} and \Cref{conj-typical}
 imply \Cref{second-main}, we deduce from 
\Cref{geometricZP}, \Cref{level criterion} and \Cref{typicallocus} the
following:

\begin{cor} \label{KO1}
Let $\VV$ be a polarizable $\ZZ$VHS on a smooth connected complex quasi-projective variety
$S$. Then either $\HL(S, \VV^{\otimes})_{\fpos}$ is a
finite union of maximal atypical special subvarieties of $S$ for
$\VV$, hence is algebraic; or the level of $\VV$ is one or two,
$\HL(S, \VV^{\otimes})_{\typ} \not = \emptyset$, and $\HL(S, \VV^\otimes)$ is analytically dense in $S$.
\end{cor} 

\begin{rmk}
  Notice that \Cref{KO1} clarifies the alternative of \Cref{KO}, and
  strengthens it, as density is now in the analytic topology rather
  than in the (weaker) Zariski topology. However it does not recover
  exactly \Cref{KO}, as in \Cref{KO1}
  the density could come from typical varieties which are not
  factorwise of positive dimension.
  \end{rmk}


\subsubsection{Examples of dense typical Hodge locus in level $1$}\label{Picard}
Given a $\ZZ$VHS on $S$ of level one or two, can we deduce from
the generic Mumford-Tate datum $(\G, D)$ whether or not the typical
Hodge locus is dense in $S$? Here we discuss a simple
example (related to \Cref{rmkchai}) where density holds and, in
the next paragraph, use it to present an arithmetic application
of independent interest.

\begin{theor}[Chai + $\epsilon$]\label{typicaldiv}
Let $(\mathbf{G},X)$ be a Shimura datum containing a Shimura sub-datum
$(\mathbf{H},X_H)$ such that $X_H$ is one dimensional and the normaliser of 
$H$ in $G$ is $H$. Assume moreover that $\mathbf{G}$ is
absolutely simple. Let $S \subset
\Gamma \backslash X$ be an irreducible subvariety of codimension one.
Then the Hodge locus of $S$ is dense.  
 \end{theor}

 \begin{rmk}\label{specailagrmk}
Regarding the Shimura variety parametrising $g$-dimensional principally polarised abelian varieties
$
\mathcal{A}_g= \Sh_{\mathbf{Sp}_{2g}(\Z)} ( \mathbf{GSp}_{2g},\mathbb{H}_g)$,
we recall that its dimension is $\frac{g(g+1)}{2}$ and the biggest dimension of the (strict) special subvarieties is 
$
\frac{g(g-1)}{2}+1$, 
realised for instance by $\mathcal{A}_{g-1}\times
\mathcal{A}_{1}\subset \mathcal{A}_g$. It follows that the typical
locus of any subvariety of $\mathcal{A}_{g}$ of dimension smaller than
$g-1$
is empty. Let $S\subset \mathcal{A}_g$, be a (closed, not
necessarily smooth) subvariety of dimension $q$,
one would expect that the typical Hodge locus is (analytically) dense in
$S$ if and only if $q\geq g-1$. 
 \end{rmk}

\begin{rmk}
We do not address here the question of whether the typical Hodge locus
is equidistributed in level $1$ or $2$. 
This interesting question has been
investigated in more details for the typical Hodge locus of zero
period dimension by Tayou and Tholozan in \cite{tayou, 2021arXiv210315717T}. In particular \Cref{typicaldiv} is also a
corollary of their work.
 See also the work of Koziarz and Maubon
\cite{km} and \cite[Proposition 3.1]{zbMATH06442355}.
\end{rmk}

 \subsubsection{An arithmetic application} \label{The curve}
In this section we present an arithmetic application of
\Cref{typicaldiv}, that was our motivation for studying such a problem.
 Mumford \cite{zbMATH03271259} shows that there exist
principally polarized abelian varieties $X$
of dimension $4$ having trivial endomorphism ring, that are not Hodge
generic in $\cA_4$ (they have an exceptional Hodge class in
$H^4(X^2, \ZZ)$). A question often attributed to Serre is to describe ``as
explicitly as possible'' such abelian varieties \emph{of Mumford's
  type}. The most satisfying way would be to show the existence of a smooth projective curve over $\overline{\QQ}$ of genus $4$,
whose Jacobian is of Mumford's type. Gross
\cite[Problem 1]{zbMATH01587304} asked the weaker question over
$\CC$. See also a related question of Oort 
\cite[Section 7]{zbMATH07128855} (which we will address at the very end of the paper, see indeed \Cref{oortrmk}). 
We also notice that, at the end of August 2021, F. Calegari has
proposed a polymath project to find an explicit example of a
Mumford 4-fold over $\Q$.

\begin{theor}\label{serrequestion}
 There exists a smooth projective curve $C/\Qbar$ of genus $4$ whose
 Jacobian has Mumford-Tate group isogenous to a $\Q$-form of the complex group $\Gm \times \Sl_2\times \Sl_2 \times \Sl_2$. 
\end{theor}

\begin{rmk}
Actually our proof shows the existence of infinitely many such
curves. Finding explicit equations for such a curve remains an open
problem.
\end{rmk}

\begin{rmk}
A crucial input in the proof
is the Andr\'{e}--Oort conjecture for $\cA_4$, as
established in \cite[Theorem 1.3]{MR3744855}. For $g=4$,
  \cite[Theorem 5.1]{MR3177266} and \cite[Theorem 1.3]{zbMATH06284346}
  are actually enough, see \Cref{rmkao} for more details.
\end{rmk}

\subsection{Complements and applications}\label{complapp}

We conclude this section by discussing two applications of  
\Cref{geometricZP} of particular interest.

\subsubsection{The Shimura locus} \label{Shimura}
Recall that a $\ZZ$VHS $\VV$ is said \emph{of
Shimura type} if its generic Hodge datum $(\G, D)$ is a Shimura
datum, see \cite[Definition 3.7]{klin}. The $\ZZ$VHSs of Shimura type form the simplest class of $\ZZ$VHSs:
the target $\Gamma \backslash D$
of the period map (\ref{period0}) is a Shimura variety (a quasi
projective algebraic variety by \cite{MR0216035}) and the period
map $\Phi: S^{\an} \to \Gamma \backslash D$ is algebraic, see
\cite[Theorem 3.10]{zbMATH03394558}.

Given $\VV$ a polarizable $\ZZ$VHS on a smooth connected complex quasi-projective variety
$S$, a special subvariety $Z \subset S$ for $\VV$ is said to be \emph{of
Shimura type} if $\VV_{|Z}$ is of Shimura type, i.e. 
its generic Hodge datum $(\G', D')$ is a Shimura datum. In that case,
it is said to have \emph{dominant period map} if the algebraic period map $\Phi_{|Z^\an}: Z^\an \to
\Gamma' \backslash D'$ is dominant. The Andr\'e-Oort conjecture for $\Z$VHS, as
formulated in \cite[page 275]{MR2918237} and \cite[Conjecture
5.2]{klin}, states that if the union of the special subvarieties of $S$ for $\VV$
which are of Shimura type with dominant period
maps is Zariski-dense in $S$, then $(S, \VV)$
itself is of Shimura type with dominant period map. This conjecture
easily follows from \Cref{main conj}, see \cite[Section
5.2]{klin}.

\medskip
More generally let us define the \emph{Shimura locus} of $S$ for $\VV$ as the
union of the special subvarieties of $S$ for $\VV$ which are of Shimura
type (but not necessarily with dominant period maps). In
\cite[Question 5.8]{klin}, the second author asked the following 
question, which generalizes the Andr\'{e}--Oort conjecture for $\Z$VHSs:
\begin{question}\label{qestionkling}
Let $\VV$ be a polarizable $\ZZ$VHS on a smooth connected complex quasi-projective variety
$S$. Suppose that the Shimura locus of $S$ for $\VV$
is Zariski-dense in $S$. Is it true that necessarily
$\VV$ is of Shimura type? 
\end{question}

Thanks to \Cref{geometricZP}, we answer affirmatively the geometric part of \Cref{qestionkling}:
\begin{cor}\label{shimuralocus}
Let $\VV$ be a polarizable $\ZZ$VHS on a smooth connected complex quasi-projective variety
$S$, with generic Hodge datum $(\G, D)$, and assume that $\G^{\der}$ is the generic Mumford-Tate group of $\VV$. Suppose that the Shimura
locus of $S$ for $\VV$ of positive period dimension is Zariski-dense in $S$.

If the adjoint group of the generic Mumford-Tate group
$\G$ of $\VV$ is simple then $(\G, D)$ is a Shimura datum.

In general, either $(\G,D)$ is a Shimura datum, or there exists a decomposition
$
\widehat{S}^{\an}\xrightarrow{(\Phi', \Phi'')}  \Gamma' \backslash D'
\times  \Gamma'' \backslash D'' \to  \Gamma \backslash D, 
$
where $\widehat{S}\to S$ is a finite surjective morphism; $ \Gamma'
\backslash D' $ is a Shimura variety; $\Phi'$ is dominant, and $\Gamma' \backslash D'
\times  \Gamma'' \backslash D'' \to  \Gamma \backslash D, 
$ is a finite Hodge morphism; $\Phi (S^{\an})$ is dense in the image of $(\Gamma' \backslash
  D'  \times \Phi''(S^{\an}))$ in $ \Gamma \backslash D$; and the infinitely many Shimura subvarieties of $\Gamma \backslash
  D$ intersecting $\Phi (S^{\an})$ are coming from the image of 
  \begin{displaymath} 
\operatorname{Shimura \ sub-variety}\times {
  \operatorname{CM-point}}\subset  \Gamma' \backslash D' \times
\Gamma'' \backslash D''  
\end{displaymath}
in $\Gamma \backslash D$.

\end{cor}

\begin{rmk}
Notice that \Cref{shimuralocus} implies the geometric part of
the Andr\'{e}--Oort conjecture for $\Z$VHS, which has the same conclusion but requires the sub-Shimura
varieties to be fully contained in $\Phi(S^{\an})$ rather than 
intersecting it in positive dimension. Such a statement has been recently proven
independently of this paper and with different techniques by Chen, Richard
and the third author \cite[Theorem A.7]{eqiIII}.
\end{rmk}

\begin{rmk}\label{rmk2}
Since the horizontal (in the sense of \Cref{newzpsection}) tangent subbundle of $\Gamma\backslash D$ is
locally homogeneous, any Hodge variety $\Gamma \backslash D$ containing one positive dimensional Shimura
subvariety (induced by a Hodge-subdatum) contains infinitely
many of such (compare with \Cref{rmk1} and \cite[Lemma
3.2]{Robles}). For a discussion regarding the horizontal
tangent bundle we refer to the beginning of \Cref{newzpsection}. The
same applies to \emph{CM-points}, that is points whose Mumford--Tate
group is commutative, which are always analytically dense in $\Gamma
\backslash D$. 
\end{rmk}

\subsubsection{The modular locus} \label{modular}
Any $\ZZ$VHS $\VV$ of Shimura type with dominant (algebraic) period map $\Phi:
S \to \Gamma \backslash D$
endows $S$
with a large collection of algebraic self-correspondences, coming from the Hecke
correspondences on the Shimura variety $\Gamma \backslash D$
associated to elements $g\in \mathbf{G}(\Q)_+$. These correspondences
are examples of {\em special correspondences} in the following sense (generalising \cite[Definition 6.2.1]{bu}):

\begin{defi}\label{specialcorrdef}   \label{modularlocusss}
Let $\VV$ be a polarizable $\ZZ$VHS on a smooth connected complex quasi-projective variety
$S$.
A \emph{special correspondence} for $(S,\VV)$ is an irreducible subvariety $W\subset S\times S$ such that:
\begin{enumerate}
\item Both projection maps $W\rightrightarrows S$ are
surjective finite morphisms;
\item $W$ is a special subvariety for $(S\times S, \VV \times \VV)$.
\end{enumerate}
The \emph{modular locus} for $(S,\VV)$ is the union in $S \times S$ of the special correspondences.
\end{defi}

\noindent
Thus the modular locus is a subset of the Hodge locus of $(S\times S,
\VV \times \VV)$. 
If $(S, \VV)$ is of Shimura type with dominant period map, this
modular locus is well-known to be Zariski-dense in $S \times S$. 
As a second corollary to \Cref{geometricZP}, we prove the converse:

\begin{cor}[Hodge commensurability criterion]\label{modularlocus}
Let $\VV$ be a polarizable $\ZZ$VHS on a smooth connected complex quasi-projective variety
$S$.
Then $(S, \VV)$ is of Shimura type with dominant period map if and
only if the modular locus for $(S, \VV)$ is Zariski-dense in $S \times S$.
\end{cor}

\begin{rmk}
 A special case of \Cref{modularlocus} has been recently proven by the
 first and third authors. They used it to give a new proof of the Margulis
 commensurability criterion for arithmeticity (for complex hyperbolic
 lattices), see \cite[Theorem 6.2.2]{bu} and the discussion in Section
 6.2 in \emph{op. cit.}.
 \end{rmk}

 \begin{rmk} \label{rmk1}
 Recall that a Hodge variety $\Gamma \backslash D$ can contain no positive
 dimensional Hodge subvariety at all, that is the only sub-Hodge datum
 $(\mathbf{H},D_H)\subsetneq (\mathbf{G},D)$ are associated to
 CM-points. This is of course the case if $\Gamma \backslash D$ has
 dimension one, but there are also higher dimensional Shimura
 varieties with the same property. The most famous example is perhaps
 given by $D= \mathbb{B}^n$ (the Hermitian symmetric space associated
 to $\PU(1,n)$), with $n+1$ a prime number and $\Gamma$ an arithmetic
 lattice associated to a simple division algebra (see indeed \cite[Example 9.2]{zbMATH07698517} for all details).
 
However $\Gamma \backslash D$ always admits infinitely many special
 correspondences. \Cref{modularlocus} says that these special
 correspondences do very rarely pullback to finite correspondences on $S$.
\end{rmk}

\subsection{Related work}

A key ingredient in the proof of \Cref{geometricZP} is the Ax-Schanuel conjectured by the second author \cite[Conjecture
7.5]{klin} and recently proved by Bakker and Tsimerman \cite{MR3958791}
(which is recalled in \Cref{zassection}). Starting from Pila and
Zannier \cite{zbMATH05292756}, a link between functional transcendence
and the Zilber--Pink conjecture has been observed in \cite{MR3177266,
  MR3552014, MR3867286}, at least for subvarieties of abelian and pure
Shimura varieties. For the more general case of $\Z$VHS it has
appeared for the first time in the proof of \cite[Theorem
1.2.1]{bu}. In retrospect even the very first proof of the so called \emph{geometric Manin-Mumford} used some functional transcendence (in the form of Bloch-Ochiai Theorem). See indeed \cite[Theorem 4]{zbMATH03908827}. 
We remark here that a new Ax-Schanuel result for foliated principal
bundles has recently been obtained by Bl\'{a}zquez-Sanz, Casale,
Freitag, and Nagloo \cite{2021arXiv210203384B}. Such an Ax-Schanuel
implies the one we are using, it provides a new proof which does not
rely on o-minimality (even if our proof of \Cref{geometricZP} still uses some light
o-minimality). 

\medskip
While this paper was being completed it seems that 
Pila and Scanlon also noticed a difference between the two Zilber--Pink
conjectures for $\Z$VHS, namely \Cref{main conj} and \cite[Conjecture
1.9]{klin} (a point which is addressed in more details in
\Cref{newzpsection}). See indeed \cite[Theorem 2.15 and Remark
2.15]{2021arXiv210505845P}. In the same paper they also speculate
about the Zariski density of the typical Hodge locus (see Conjecture
5.10 in \emph{op. cit.}). It could be that the strategy presented here
applies also in their function field versions. 

\medskip
While this paper was being completed, we learned that de Jong
\cite{dj} had also wondered about \Cref{main conj}. It is remarkable
that he had considered such a problem even before the Zilber--Pink
conjecture for Shimura varieties happened to be in print. In the same
document he discusses the difference between \Cref{main conj} and the
Zilber--Pink conjecture as stated in \cite[Conjecture
1.9]{klin}, and asks if \Cref{level criterion} could
be true. We thank him for sharing with us his insightful notes.

\subsection{Outline of the paper}
\Cref{preliminaries} contains some preliminaries on variational Hodge
theory, most notably discussing the notion of special and
weakly special subvariety, the level of a $\ZZ$VHS and functional transcendence properties of
period maps. \Cref{section4} describes in details the notions of
atypicality and our variant of
the Zilber--Pink conjecture for $\Z$VHS, presenting equivalent
formulations and applications. \Cref{atypicalsection} is devoted to
atypical intersections, proving the geometric part of the Zilber--Pink
conjecture \Cref{geometricZP}, and the fact that families of maximal
atypical weakly special subvarieties lie in typical intersections.
\Cref{criterionsection} proves 
\Cref{level criterion}: in level at least 3, every special subvariety is
atypical. We combine \Cref{geometricZP} and \Cref{level
  criterion} to obtain \Cref{corol}  in \Cref{proof corol}. In \Cref{newzpandtyp}, we prove the results announced in
\Cref{mainsectionintro} and \Cref{complapp}.
\Cref{typicalsection} proves the
``all or nothing'' principle \Cref{typicallocus} for typical intersections. 
Finally \Cref{examplesection}
discusses the density of the typical Hodge locus in level one, and \Cref{finalsection} 
concludes the paper presenting our application to the existence of
  curves of genus four with a given Mumford--Tate group (Serre-Gross' question).

\subsection{Acknowledgements}
The authors would like to thank S. Tayou and N. Tholozan for pointing out a mistake related to
\Cref{specailagrmk} in a previous version;
S. Tayou and T. Kreutz for their careful reading and comments on a first version
of this paper; C. Robles for pointing out a mistake in a previous
version and sharing her thoughts on Section 7.1;  B. Farb for sharing the reference
\cite{zbMATH04057672}; and M. Green,
P. Griffiths and C. Robles for their interest in
this work. Finally we thank the referees for their careful reading,
which improved the paper.

\section{Preliminaries}\label{preliminaries}
In this section we recall the definitions and results from Hodge
theory we will need, as well as adequate references for more details.

\subsection{Some notation}\label{conventions}
\begin{itemize}
  \item In this paper an algebraic variety $S$ is a reduced scheme of
    finite type over the field of complex numbers, not
    necessarily irreducible;
\item If $S$ is an algebraic (resp. analytic) variety, by a subvariety
  $Y \subset S$ we always mean a \emph{closed} algebraic
  (resp. analytic) subvariety. Its smooth locus is denoted by
  $Y^{\operatorname{sm}}$;
  \item Given $\G$ an algebraic group, we denote by $\G^\ad$ its
    adjoint group and by $\G^\der$
    its derived group. Similarly if $\fg$ is a Lie algebra we denote
    by $\fg^\ad$ its adjoint Lie algebra and by $\fg^\der = [\fg,
    \fg]$ its derived Lie algebra. When $\G$ (resp. $\fg$) is reductive the natural morphism
    $\G^\der \to \G^\ad$ (resp. $\fg^\der \to \fg^\ad$) is an isogeny
    (resp. an isomorphism). If $\G$ is a $\QQ$-algebraic group we
    denote by $G$ the connected component of the identity $\G(\RR)^+$
    of the real Lie group $\G(\RR)$. We use the index $_+$ to denote
    the inverse image of $\G^{\ad}(\RR)^+$ in $\G(\R)$. Finally we set
    $\G(\Q)^+=G\cap \G(\Q)$ and $\G(\Q)_+=\G(\RR)_+\cap \G(\Q)$. 
    If $\HH \subset \G$ is an
    algebraic subgroup we denote by $\mathbf{Z}_\G(\HH)$ its
    centraliser in $\G$ and by $\mathbf{N}_\G(\HH)$ its normaliser;
  \item We refer to \cite{Voisin} for standard definitions on Hodge
    structures and variations of Hodge structures. We denote by $\bS$ the Deligne torus $\Res_{\CC/\RR}
      \Gm$. Given $R= \ZZ$, $\QQ$ or $\RR$ we recall that an $R$-Hodge
      structure on a finitely generated $R$-module
      $V$ is equivalent to the datum of a morphism of $\RR$-algebraic group $h: \bS \to
      \GL(V_\RR)$, where $V_\RR:= V \otimes_R \RR$. For $R = \ZZ$
      or $\QQ$ a Hodge class for $V$ is a $\bS$-invariant vector $v
      \in V_\QQ$. 
      We remark here that we do not need to assume that $h$ restricted to $\Gm\subset \bS$ is defined over $\QQ$. 
      A Hodge tensor for $V$ is
      a Hodge class for $\bigoplus_{a, b \in \ZZ} V^{\otimes a} \otimes
      (V^\vee)^{\otimes b}$ where $V^\vee$ denotes the Hodge structure
      dual to $V$. In this paper all Hodge structures and variations
      of Hodge structures are polarizable ones;
\item The category of $\QQ$-Hodge-structures is Tannakian. Recall that the Mumford--Tate group
$\MT(V) \subset \GL(V)$ of a $\QQ$-Hodge structure  
$V$ is the Tannakian group of the Tannakian subcategory $\langle
V\rangle ^\otimes$ of $\QQ$-Hodge structures
generated by $V$.
Equivalently, it is the smallest $\QQ$-algebraic subgroup of
$\GL(V)$ whose base-change to $\RR$ contains the image of $h: \bS \to
\GL(V_\RR)$. It is also the fixator in $\GL(V)$
of the Hodge tensors for $V$. As $V$ is polarised, this is a reductive group. We recall here that $\operatorname{int}(h(i))$
 is a Cartan involution of $\MT(V)_\R/h(\G_{\textnormal{m}, \R)})$.
See \cite{zbMATH03271259}, and \cite[Section 2.1]{Moonen} for more details;

\item We remark here that the main results
  of the paper could also apply to the case of admissible,
  graded-polarizable variation of mixed Hodge structures.

\end{itemize}

\subsection{Generic Mumford--Tate group and algebraic monodromy
  group}

Let $\VV$ be a $\ZZ$VHS on a smooth quasi-projective
  variety $S$ and $Y\subset S$ a closed irreducible algebraic
  subvariety (possibly singular).

A point $s$ of $Y^{\an}$ is said to be Hodge-generic
in $Y$ for $\VV$ if $\MT(\VV_{s, \QQ})$ has maximal dimension when $s$
ranges through $Y^{\an}$. Two Hodge-generic points in $Y^{\an}$ for $\VV$
have isomorphic Mumford-Tate group, where an isomorphism is given by 
horizontal transport along a path between the two points, called  \emph{generic Mumford-Tate group}
$\G_Y= \MT(Y, \VV_{|Y})$ of $Y$ for $\VV$. The Hodge locus $\HL(S, \VV^\otimes)$ is
also the subset of points of $S$ which are not Hodge-generic in $S$
for $\VV$.

The algebraic monodromy group $\HH_Y$ of $Y$ for
$\VV$ is the identity component of the Zariski-closure in $\GL(V_\QQ)$ of the monodromy of the
restriction to $Y$ of the local system $\VV_\ZZ$. It follows from
Deligne's ``Theorem of the fixed part'' and ``Semisimplicity
Theorem'' \cite[Section 4]{MR0498551} that $\HH_Y$ is a 
normal subgroup of the derived group $\G_Y^{\textnormal{der}}$, see
\cite[Theorem 1]{MR1154159}.

A simple example where the inclusion
$\HH:= \HH_S \subset \G^{\der}$ is strict is provided by a $\ZZ$VHS of
the form $\VV = \VV_1 \otimes V_2$, where $\VV_1$ is a
$\ZZ$VHS on $S$ with infinite monodromy and $V_2$ is a non-trivial constant $\ZZ$VHS
(identified with a single non-trivial $\ZZ$-Hodge structure).

\subsection{Period map and Hodge varieties}\label{periodmaps}
Let $\G:= \G_S$ be the generic Mumford-Tate group of $S$ for
$\VV$. Any point $\tilde{s}$ of the universal cover
  $\widetilde{S^{\an}}$ of $S^\an$
defines a morphism of real algebraic groups $h_{\tilde{s}}: \bS \to \G_\RR$.
All such morphisms belong to the same connected component $D= D_S$ of a $\G(\RR)$-conjugacy
class in $\Hom (\bS, \G_\RR)$, which has a natural structure
of complex analytic space (see \cite[Proposition 3.1]{klin}). The space
$D$ is a so-called \emph{Mumford-Tate domain}, a refinement of the classical
period domain for $\VV$ defined by Griffiths. The pair $(\G, D)$ is a connected Hodge datum in the 
sense of \cite[Section 3.1]{klin}, called the generic Hodge datum of
$\VV$. The $\ZZ$VHS $\VV$ is entirely described by its holomorphic period map
$$\Phi: S^{\an} \to \Gamma \backslash D\;\;.$$ Here 
$\Gamma \subset \G(\ZZ)$ is a finite index subgroup and $\Gamma
\backslash D$ is the associated connected \emph{Hodge 
variety} (see \cite[Definition 3.18 and below]{klin}). We denote by
$\tilde{\Phi}: \widetilde{S^{\an}} \to D$ the lift of $\Phi$ at the
level of universal coverings.


\medskip
Enlarging $S$ if necessary by adding the components of a normal crossing divisor
at infinity around which the monodromy for $\VV$ is finite, we obtain
a period map $\Phi': S' \to \Gamma \backslash D$ extending $\Phi$  which is proper (see
\cite[Theorem 9.5]{Griffiths}). Replacing first $S$ by a finite \'etale covering if
necessary, we can moreover assume that the arithmetic group $\Gamma$ is neat (in
particular torsion-free, thus $\Gamma \backslash D$ is a smooth
complex analytic variety); that the monodromy at infinity of $\VV$ is
unipotent \cite[Lemma (4.5)]{sch73}. We leave it to the reader to
verify that the results of the introduction are true for the original
variety if and only if they are true for the modified variety. As a
result in the following we will always assume that these conditions are satisfied.

\subsection{Special and weakly special subvarieties}

For this section we refer to \cite[Section 4]{KO} for more details.

Any connected Hodge variety $\Gamma \backslash D$ is naturally endowed with a countable collection of
irreducible complex analytic subvarieties: its special
subvarieties $\Gamma_{\G'} \backslash D' \subset \Gamma \backslash D$
for $(\G', D') \subset (\G, D)$ a Hodge
subdatum, and $ \Gamma_{\G'}= \G'(\Q)_+\cap \Gamma$. More generally one defines the notion of weakly special
subvariety of a Hodge variety. Let $\Gamma \backslash D$ be a Hodge
variety, with associated connected Hodge datum $(\G, D)$. A weakly
special subvariety of the Hodge variety $\Gamma
\backslash D$ is either a special subvariety or a subvariety image of $$\Gamma_\HH \backslash
D_H \times \{t\} \subset  \Gamma_\HH \backslash
D_H\times \Gamma_\LL \backslash
D_L \stackrel{f}{\to}\Gamma \backslash D\;\;,$$ where $(\HH \times \LL, D_H
\times D_L)$ is a Hodge subdatum of $(\G^\ad, D)$, $\{t\}$ is a Hodge
generic point in $\Gamma_\LL \backslash
D_L$ and $f$ is a finite morphism of Hodge varieties. The datum $((\HH, D_H), t)$ is called a \emph{weak Hodge subdatum} of
$(\G, D)$.

\begin{rmk} \label{converse}
Any special subvariety of $\Gamma \backslash D$ is weakly
special but the converse does not
hold: the simplest example of a weakly special subvariety which is not
special is provided by a Hodge generic point in a Shimura
variety. Notice however that a weakly special subvariety of $\Gamma \backslash
D$ containing a special subvariety $\Gamma_{\G'} \backslash D' \subset
\Gamma \backslash D$ is special: with the notation above, the existence
of the morphism of Hodge data $(\G', D') \to (\HH \times \LL, D_H
\times D_L)$ sending $\Gamma_{\G'} \backslash D'$ to $\Gamma_\HH \backslash
D_H \times \{t\}$ forces $\LL$ to be a torus and $t$ to be CM-point.
\end{rmk}

Let $Y \subset S$ be a closed irreducible algebraic subvariety, with
generic Mumford-Tate group $\G_Y$ and algebraic monodromy group $\HH_Y$
for $\VV_{|Y}$. Suppose that $\HH_Y$ is a strict normal subgroup of
$\G_Y^\der$. The adjoint group $\G_Y^\ad$ decomposes as a non-trivial product
$\HH_Y^\ad \times \LL_Y$. Let $\widetilde{Y^{\an}} \subset \widetilde{S^{\an}}$ be an
irreducible complex analytic component of the preimage of $Y^\an$ in
$\widetilde{S^{\an}}$. The image $\tilde{\Phi}(\widetilde{Y^{\an}})$
is contained in a unique closed
$\G_Y(\RR)^+$-orbit $D_{G_{Y}} = D_{H_{Y}} \times D_{L_{Y}}$ (a
Mumford-Tate domain for $\G_Y$) in $D$, and in a unique closed 
closed $\HH_Y(\RR)^+$-orbit $D_{H_{Y}} \times \{\tilde{t}_Y\}$
(a weak Mumford-Tate domain) in $D$. This
gives rise (replacing $Y$ by a finite \'etale cover if necessary) to a factorization
\begin{displaymath}
  \Phi_{|Y^{\an}}: Y^{\an} \to \Gamma_{\HH_{Y}} \backslash 
    D_{H_{Y}} \times \{t_Y\} \hookrightarrow \Gamma_{\G_{Y}}
    \backslash D_{G_{Y}} = \Gamma_{\HH_{Y}} \backslash
    D_{H_{Y}} \times \Gamma_{\LL_{Y}} \backslash D_{L_{Y}} \subset \Gamma \backslash D \; ,
\end{displaymath}
where $\Gamma_{\HH_{Y}} \backslash D_{H_{Y}} \times \{t_Y\} $
(resp. $\Gamma_{\G_{Y}} \backslash D_{\G_{Y}}$) is the smallest weakly
special (resp. special) subvariety of $\Gamma \backslash D$ containing
$\Phi(Y)$, called the weakly special closure (resp. the special
closure) of $Y$ in $\Gamma \backslash D$. In the case where $\HH_Y =
\G_Y^\der$ the weakly special closure and the special closure of $Y$
in $\Gamma \backslash D$ coincide.

\begin{nota} To simplify the notation we will often write hereafter simply 
\begin{displaymath} \label{period map}
  \Phi_{|Y^{\an}}: Y^{\an} \to \Gamma_{\HH_{Y}} \backslash 
    D_{H_{Y}} \subset \Gamma \backslash D 
\end{displaymath}
for the period map for $Y$, calling $(\HH_Y, D_{H_{Y}})$ the weak
Hodge datum associated to $Y$ and specifying $t_Y$ and $\LL_Y$ only when we
need them.
\end{nota}

\begin{rmk}
  In particular, with the notation of \ref{period},
    $D_{\HH_{}}= D_1 \times \cdots \times D_r$, $\Gamma_{\HH_{S}}=\Gamma_1
    \times \cdots \times \Gamma_r$, and $\{t_S\}$ is the unique point
    of $D'$ determined by the projection of $\Phi(S^{\an})$.
  \end{rmk}

\medskip
\Cref{CDK} can be rephrased by saying that the preimage under $\Phi$
of any special subvariety of $\Gamma \backslash D$ is an algebraic
subvariety of $S$. More generally the preimage under $\Phi$
of any weakly special subvariety of $\Gamma \backslash D$ is an algebraic
subvariety of $S$, see \cite[Corollary 4.8]{KO}.

\begin{defi}\label{defsp}
An irreducible algebraic subvariety $Y \subset S$ is
called \emph{special}, resp. \emph{weakly special} for $\VV$ or $\Phi$ if it is an irreducible component
of the $\Phi$-preimage of its special closure (resp. its weakly
special closure) in $\Gamma \backslash D$. 
\end{defi}

\noi
The Hodge locus $\HL(S, \VV^\otimes)$ is
then the countable union of the strict special subvarieties of $S$ for
$\VV$. Notice, on the other hand, that any point of $S$ is weakly
special for $\VV$. We thus define:

\begin{defi}
The \emph{weakly special locus} $\HL(S, \VV^\otimes)_\ws$ of $S$ for $\VV$ is the
union of the strict weakly special subvarieties of $S$ for $\VV$ 
  of positive period dimension.
\end{defi}

\noi
In particular : $\HL(S, \VV^\otimes)_\pos \subset \HL(S, \VV^\otimes)_\ws$.

\begin{rmk} \label{reduction}
For studying $\HL(S, \VV^\otimes)$, $\HL(S, \VV^\otimes)_{\ws}$,
$\HL(S, \VV^\otimes)_\pos$, or $\HL(S, \VV^\otimes)_{\fpos}$, we can
without loss of generality assume that the algebraic monodromy group
$\HH:= \HH_S$ of $S$ for $\VV$ coincide with $\G^{\der}$, hence $\Gamma_\HH \backslash D_H =\Gamma
\backslash D$. Indeed if $\HH$ is a strict normal subgroup of $\G^\der$ the morphism of Hodge datum
$$(\G, D) \to (\G^\ad = \HH^\ad\times \LL, D_H \times
D_L) \stackrel{p_1}{\to} (\HH^\ad, D_H)$$ induces a commutative
diagram of period maps
$$
\xymatrix{S^{\an} \ar[r]^>>>>>\Phi \ar[drr]_{\Phi'} &\Gamma_\HH \backslash D_H \times
\{t_S\} \ar@{^(->}[r] & \Gamma \backslash D \ar[d]^{p_1} \\
& & \Gamma_\HH \backslash D_H} \;\;.
$$ One easily checks that the special (resp. weakly special) subvarieties of $S$ for
$\Phi$ coincide with the special (resp. weakly special) subvarieties of $S$ for $\Phi'$.
\end{rmk}

As proven in \cite[Corollary 4.14]{KO} one
obtains the following equivalent definitions of special and weakly special
subvarieties of $S$ for $\VV$:

\begin{lem} \label{special}
The  special subvarieties of $S$ for $\VV$ are the closed irreducible
algebraic subvarieties $Y \subset S$ maximal among the closed
irreducible algebraic subvarieties $Z$ of $S$ such that the generic
Mumford-Tate group $\G_Z$ of $Z$ for $\VV$ equals $\G_Y$.
\end{lem}

\begin{lem} \label{equi}
The weakly special subvarieties of $S$ for $\VV$
are the closed irreducible
algebraic subvarieties $Y \subset S$ maximal among the closed
irreducible algebraic subvarieties $Z$ of $S$ whose algebraic
monodromy group $\HH_Z$ with respect to $\VV$ equals $\HH_Y$.
\end{lem}

\subsection{Hodge-Lie algebras} \label{Hodge-Lie}

\begin{defi} \label{HL-algebra}
  Let $K= \QQ$ or $\RR$, or a totally real number field. If $K=\Q$ or $\R$, a $K$-\emph{Hodge-Lie algebra} 
  is a reductive (finite dimensional) Lie
  algebra $\fg$ over $K$ endowed with a $K$-Hodge structure of weight zero
  $$\fg_\CC = \bigoplus_{i \in \ZZ} \fg^i, \qquad  \textnormal{and}
  \quad \fg^{-i}
  =\overline{\fg^i} \quad \forall i \in \ZZ\;\;,$$
  such that the Lie bracket $[\cdot, \cdot]: \bigwedge^2 \fg \to  \fg$ is
  a morphism of $K$-Hodge structures, and the negative of the Killing form
  $B_\fg: \fg^{\ad} \otimes \fg^{\ad} \to K$ is a polarisation of
  the $K$-Hodge substructure $\fg^\ad$ (where we identify the adjoint
  Lie algebra $\fg^\ad$ with the derived one $\fg^\der:=[\fg, \fg]$).

 If $K$ is a totally real number field a $K$-\emph{Hodge-Lie algebra} is the
 datum of a reductive Lie algebra $\fg$ over $K$ and of a
 $\QQ$-Hodge-Lie algebra structure on $\Res_{K/\QQ} \fg$.
\end{defi}

\begin{rmk}
  In \Cref{HL-algebra} the notation $\fg^i$ abbreviates the classical
  $\fg^{i, -i}$ of Hodge theory.
\end{rmk}

\begin{rmk} \label{remfactor}
If $\fg$ is a $K$-Hodge-Lie algebra then the
derived Lie algebra $\fg^\der$ is a $K$-Hodge-Lie subalgebra
of $\fg$,
which is equal to the
adjoint $K$-Hodge-Lie algebra $\fg^\ad$ quotient of $\fg$ by its
center; if $\fg$ is semi-simple as a Lie algebra, then its simple
factors are naturally $K$-Hodge-Lie algebras; and being simple
as a $K$-Lie algebra is equivalent to being simple as a $K$-Hodge-Lie algebra.
\end{rmk}

\begin{rmk} \label{complex}
Given $h : \bS \to \G_\RR $ a pure Hodge structure on a real
  algebraic group $\G_\RR$ in the sense of 
  \cite[(page 46)]{Simpson}, the adjoint action of $h$ on the Lie algebra
  $\fg_\RR$ endows $\fg_\RR$ with the structure of a real Hodge-Lie
  algebra. Conversely one easily checks that a real Hodge-Lie algebra
  structure on $\fg_\RR$ integrates into a Hodge structure on some
  connected real algebraic group $\G_\RR$ with Lie algebra
  $\fg_\RR$. In particular if a simple real Lie algebra $\fg_\RR$
  admits a Hodge-Lie structure its complexification $\fg_\CC$ is still
  simple, see \cite[4.4.10]{Simpson}.
\end{rmk}

\subsection{Level}

\begin{defi} \label{levelHS}
Given a simple $\RR$-Hodge-Lie algebra
  $\fg_\RR$, its
{\em level} is the largest integer $k$ such that $\fg^{k} \not =
0$. The \emph{level}
 of a simple $\QQ$-Hodge-Lie algebra $\fg$ is
the \emph{maximum} of the level of the irreducible factors of $\fg_\RR:= \fg
\otimes_\QQ \RR$. The \emph{level} of a semi-simple $\QQ$-Hodge-Lie algebra is the
\emph{minimum} of the levels of its simple factors. If $K= \Q$ or
$\R$, the \emph{level} of a $K$-Hodge structure $V$ is the level of
its adjoint $K$-Hodge-Lie Mumford-Tate algebra $\fg^\ad$.
\end{defi}

\begin{rmk}
Notice that our \Cref{levelHS} is 
\emph{not} the standard one. Usually, the level of any pure real
Hodge structure $V$ is defined as the maximum of $k-l$ for
$V^{k,l}\not =0$. We believe that our definition, which takes into
account only the adjoint Mumford-Tate Lie algebra of $V$, is more fundamental. For instance, the level of the weight $2$ Hodge structure
$H^2(S, \QQ)$, for $S$ a $K3$-surface, is one for \Cref{levelHS},
reflecting the fact that the motive of $S$ is of abelian type, while
it would be two with the usual definition.
\end{rmk}

  It follows from \Cref{complex} that if $\VV$ is a polarizable $\ZZ$VHS
  on a smooth connected quasi-projective variety $S$, with generic Hodge datum
  $(\G, D)$ and algebraic monodromy group $\HH$, then any Hodge generic point $x\in \tilde{\Phi}(\widetilde{S^{\an}})$
  defines a $\QQ$-Hodge-Lie algebra
  structure $\fg^\ad_x$ on $\fg^\ad$, with $\QQ$-Hodge-Lie subalgebra $\fh_x$. One
  immediately checks that the levels of these two structures are
  independent of the choice of the Hodge generic point $x$ in
  $D$.

\begin{defi} \label{levelVHS}
Let $\VV$ be a polarizable $\ZZ$VHS
on a smooth connected complex quasi-projective variety 
$S$, with algebraic monodromy group $\HH$. The \emph{level} of $\VV$ is the
level of the $\QQ$-Hodge-Lie algebra $\fh_x$ for $x$ any Hodge
generic point of $\tilde{\Phi}(\widetilde{S^{\an}})$ .
\end{defi}

\subsection{Functional Transcendence}\label{zassection}
Let $S$ be a smooth complex quasi-projective variety supporting a
polarizable $\Z$VHS $\mathbb{V}$, with generic Hodge datum $(\G, D)$, 
algebraic monodromy group $\HH$, and period map $\Phi: S \to \Gamma_{\HH}
\backslash D_{H} \subset \Gamma \backslash D$. Without loss of
generality we can assume, replacing if necessary $S$ by a finite
\'etale cover, that $\Gamma$ is torsion-free. Let $\tilde{\Phi}:
\widetilde{S^\an} \to D_H$ be the lift of $\Phi$. The domain $D_H$ is canonically embedded as an
open complex analytic real semi-algebraic subset in a flag variety
$D_H^{\vee}$ (called its \emph{compact dual}). Following \cite{klin} we define
an irreducible algebraic subvariety of $D_H$ (resp. $S \times D_H$) as a complex analytic
irreducible component of the intersection of an algebraic subvariety
of $D_H^{\vee}$ (resp. $S \times D_{H}^{\vee}$) with $D_H$ (resp. $S
\times D_{H}$).

The following result is the so called Ax-Schanuel Theorem. It was conjectured by the second author \cite[Conjecture
7.5]{klin} and later proved by Bakker and Tsimerman \cite[Theorem
1.1]{MR3958791}, generalising the work of Mok, Pila and Tsimerman
\cite[Theorem 1.1]{as} from level one to arbitrary levels.

\begin{theor}[Bakker--Tsimerman]\label{astheorem}  
Let $W \subset S \times D_H$ be an algebraic subvariety. Let $U$ be an
irreducible complex analytic component of $W \cap S \times_{\Gamma_\HH \backslash D_H} D_H$
such that 
\begin{displaymath}
\codim_{S \times D_{H}} U< \codim_{S \times D_{H}} W+ \codim_{S
  \times D_{H}} (S \times_{ \Gamma_{\HH} \backslash D_{H}}
D_H)\;\;.  
\end{displaymath}
Then the projection of $U$ to $S$ is contained in a strict weakly
special subvariety of $S$ for $\VV$.
\end{theor}
Notice that $S \times_{ \Gamma_{\HH} \backslash D_{H}} D_H$ is simply the image of the graph of $\tilde{\Phi}:
\widetilde{S^\an} \to D_H$ under $
\widetilde{S^\an} \times D_H\to S \times D_H$.
\begin{rmk}
  Following the notation of \Cref{periodmaps}, the intersection $W
  \cap S \times_{\Gamma_\HH \backslash D_H} D_H$ can be identified
  with the intersection in $S \times D_H$ between $W$ and the the
  image of $ \widetilde{S^{\an}}$ in $S \times D_H$ along the map
  $(\pi, \tilde{\Phi})$.
  \end{rmk}

For the applications presented in our paper, the (simpler) case of
$W=W'\times W''$ is enough, where $W'$ is an algebraic subvariety of
$S$, and $W''$ an algebraic subvariety of $D_H$. Notice also that in the
conclusion of \Cref{astheorem} the projection of $U$ to $S$ can be
zero dimensional. 

\begin{rmk}\label{remkarkmixed} A version of the Ax-Schanuel
  conjecture for graded-polarized admissible variations of mixed Hodge structures has
  recently been established by Chiu \cite[Theorem
  1.2]{2021arXiv210110968C} and, independently, in joint work of Gao and the second
  author \cite[Theorem 1.1]{2021arXiv210110938G}.
  On a somewhat
  different direction, the first and third author \cite[Theorem
  1.2.2]{bu} proved an Ax-Schanuel conjecture for quotients of
  Hermitian symmetric spaces by irreducible \emph{non-arithmetic}
  lattices. 
\end{rmk}

\section{Typicality versus atypicality and a strong Zilber--Pink
  conjecture}\label{section4} 

Even if our paper is focused on the case of \emph{pure} $\Z$-VHS, we
remark that everything in this section can be translated
to the more general case of graded-polarized admissible mixed $\Z$VHS (and probably 
all the results hold true, using \Cref{remkarkmixed}).

\subsection{Two notions of typicality and atypicality}\label{typatyp}
The choice to consider either the generic Mumford-Tate
group of a $\ZZ$VHS or its algebraic monodromy group, which leads to distinguish
special subvarieties from the more general weakly special subvarieties, also
leads to distinguish two notions of atypicality:

\begin{defi} \label{codimension}
Let $\VV$ be a polarizable $\ZZ$VHS on a smooth complex quasi-projective
variety $S$ and $\Phi: S^{\an} \to \Gamma_\HH \backslash D_H \subset
\Gamma \backslash D$ its period map. Let $Y \subset
S$ be a closed irreducible algebraic subvariety, with weakly special
and special
closures $\Gamma_{\HH_{Y}} \backslash D_{H_{Y}} \subset
\Gamma_{\G_{Y}} \backslash D_{\G_{Y}}$ in $\Gamma \backslash D$ . We define
the \emph{Hodge codimension} of $Y$ for $\VV$ as
$$\Hcd(Y,\VV): = \dim D_{\G_{Y}} - \dim \Phi(Y^{\an})\;\;,$$ and the \emph{monodromic
codimension} of $Y$ for $\VV$ as
$$\Mcd(Y, \VV) := \dim D_{H_{Y}} - \dim \Phi(Y^{\an})\;\;.$$
\end{defi}

Similarly:
\begin{defi} \label{atypical-general}
The subvariety $Y\subset S$ is said to be \emph{atypical} for $\VV$ if
either $Y$ is singular for $\VV$, or if $\Phi(S^\an)$
has an excess intersection in $\Gamma \backslash D$ with the special
subvariety $\Gamma_{\G_{Y}} \backslash D_{\G_{Y}}$:  
$$
\Hcd(Y, \VV) < \Hcd(S, \VV)\;\;,
$$
that is
$$\codim_{\Gamma\backslash D}
    \Phi(Y^{\an}) < \codim_{\Gamma\backslash D} \Phi(S^{\an}) +
    \codim_{\Gamma\backslash D} \Gamma_{\G_{Y}}\backslash
    D_{\G_{Y}} \;\;.
 $$
\noi
Otherwise $Y$ is said to be {\em typical}.
\end{defi}

 \begin{defi}
The subvariety $Y$ is said to be \emph{monodromically atypical} for
$\VV$ if either $Y$ is singular for $\VV$, or if $\Phi(S^\an)$
has an excess intersection in $\Gamma_{\HH} \backslash D_{H}$ with
the weakly special subvariety $\Gamma_{\HH_{Y}} \backslash D_{H_{Y}}$:
$$ \Mcd(Y, \VV) <
\Mcd(S, \VV)\;\;,$$ that is:
$$ \codim_{\Gamma_{\HH}\backslash D_{H}}
    \Phi(Y^{\an}) < \codim_{\Gamma_{\HH}\backslash D_{H}} \Phi(S^{\an}) +
    \codim_{\Gamma_{\HH} \backslash D_{H}} \Gamma_{\HH_{Y}}\backslash
    D_{H_{Y}}\;\;.$$
\noi Otherwise $Y$ is said to be {\em monodromically typical}.
   \end{defi}

Let us generalize \Cref{atypical locus} to the weakly special locus:

\begin{defi}\label{defws}
The \emph{monodromically atypical weakly special Hodge locus} $\HL(S, \VV^{\otimes})_{\ws, \watyp}$ of $S$ for $\VV$
is the union of  the monodromically atypical weakly special
subvarieties of $S$ for $\VV$ of positive period dimension.
\end{defi}

Although we won't use it in this paper, the notion of atypicality can be refined as in
\cite[Definition 1.8]{klin}:
\begin{defi}
An irreducible subvariety $Y\subset S$ is \emph{optimal} (resp. \emph{weakly
optimal}) for $\VV$
if for any irreducible subvariety $ Y \subset S$
containing $Y'$ strictly, the following inequality holds true 
\begin{displaymath}
\Hcd(Y,\VV)<\Hcd(Y',\VV) \qquad (\textnormal{resp.} \; \Mcd(Y,\VV)<\Mcd(Y',\VV)\;\;)\;\;.
\end{displaymath}
\end{defi}

\subsection{Atypicality versus weak atypicality}

Let us now compare atypicality with weak atypicality.

\begin{lem} \label{HatypisMatyp}
Any atypical subvariety is monodromically atypical, and therefore any weakly
typical subvariety is typical.
\end{lem}
\begin{proof}
Writing as above $\G_S^\ad = \HH_S^\ad \times \LL_S$ and $\G_Y^\ad = \HH_Y^\ad
\times \LL_Y$ (the groups $\LL_S$ and $\LL_Y$ being possibly trivial) the inequality 
\begin{displaymath}
\Hcd(Y, \VV) <\Hcd (S, \VV)
\end{displaymath}
of Hodge codimensions can be rewritten in terms of
monodromic codimensions as
\begin{equation} \label{passage}
  \Mcd(Y, \VV) < \Mcd (S, \VV) + \dim D_{\LL_{S}} - \dim
  D_{\LL_{Y}}\;\;.
\end{equation}

\noi
If follows immediately that $ \Mcd(Y, \VV) < \Mcd (S, \VV)$ holds true in the case
where $\dim D_{\LL_{S}} =0$, that is if $\HH_S = \G_S^{\der}$. 
In the general case notice that the morphism of
Hodge data $(\G_Y, D_Y) \to (\G_S, D_S)$ induces by projection a morphism of Hodge
data $(\G_Y , D_Y) \to (\LL_S, D_{\LL_{S}})$. As $\tilde{\Phi}(\widetilde{S^{\an}})$ is
contained in the weak Mumford-Tate subdomain $D_{\HH_{S}} \times \{t\}
\subset D_{\HH_{S}} \times D_{\LL_{S}}$, this morphism sends
$D_{\LL_{Y}}$ to a Mumford-Tate subdomain of $D_{\LL_{S}}$ containing
the point $t$. As $t$ is Hodge generic in $D_{\LL_{S}}$ it follows
that $D_{\LL_{Y}}$ surjects onto $D_{\LL_{S}}$. Thus one always has
$\dim D_{\LL_{S}} \leq \dim
D_{\LL_{Y}}$, hence $ \Mcd(Y, \VV) < \Mcd (S, \VV)$ from (\ref{passage}).
\end{proof}

The converse of \Cref{HatypisMatyp} is not true:

\begin{example} \label{HtypMatyp}
  There exists a $\ZZ$VHS $\VV$ on a smooth quasi-projective surface
  $S$ and a special curve $Y \subset S$ for $\VV$ which is
  typical but monodromically atypical.
\end{example}

\begin{proof}
Let us denote by $Y_1 = \Gamma_1(7)
\backslash \mathbb{H}$ the classical modular curve of level $\Gamma_1(7)$, where $\mathbb{H}$
denotes the Poincar\'e half-plane (the level plays no particular role in the sequel, but it is needed to have a torsion free lattice). 
We will construct $S$ as a surface
contained in the Shimura variety $\Gamma \backslash D = Y_1^3$.
To do so let us fix a CM-point $x \in Y_1$, that is a point with Mumford-Tate
 group a torus $\TT$. Let $t \in Y_1$ be a Hodge
generic point. We define
$$Y:= Y_1 \times \{t\} \times \{x\}
\subset Y_1^3 \;\;.$$
Let us choose an irreducible algebraic curve $C \subset Y_1^2$
such that the intersection $C \cap (Y_1 \times\{x\})$ is a finite
set of points, containing the point $\{t\} \times \{x\} \in
Y_1^2$. Without loss of generality we can assume that (the
compactification of) $C$ in (the compactification $X_1^2$ of)
$Y_1^2$ is of degree high enough that the algebraic monodromy group
of $C$ is $\SL_2\times \SL_2$. Finally define
$$ S := Y_1 \times C \subset Y_1^3\;\;.$$

The surface $S$ is Hodge-generic in the Shimura variety $Y_1^3$, its
generic Mumford-Tate group $\G$ is $\GL_2 \times \GL_2 \times \GL_2$, and its algebraic
monodromy group is $\SL_2 \times \SL_2 \times \SL_2$. Its Hodge codimension is $1$, its
monodromic codimension is $1$.

The curve $Y \subset S$ is an irreducible component of the intersection of $S$ with
the special surface $Y_1^2 \times \{x\}$ of $Y_1^3$, hence is a
special curve in $S$ (although it is only a weakly special curve in
$Y_1^3$). Its generic Mumford-Tate group is $\GL_2 \times \GL_2 \times \TT$, its
algebraic monodromy group is $\SL_2 \times \{1\} \times \{1\}$. Its Hodge codimension is $1$,
its monodromic codimension is zero.

Hence the special curve $Y \subset S$ is typical but monodromically atypical.
\end{proof}

\begin{rmk} \label{reduction2}
Although \Cref{HtypMatyp} shows one has to be careful when using the notion of
(a)typicality, one can without loss of generality assume
that $\HH = \G^\der$ when studying $\HL(S, \VV)_\atyp$ and
$\HL(S, \VV)_\typ$. More precisely, with the notations of
\Cref{reduction}, we claim that a special (resp. weakly special) subvariety of $S$
for $\Phi$ is atypical (resp. monodromically atypical) if and only if
it is for $\Phi'$. Indeed let us write as above $\G_S^\ad = \HH_S^\ad \times \LL_S$ and $\G_Y^\ad = \HH_Y^\ad
\times \LL_Y$ (the groups $\LL_S$ and $\LL_Y$ being possibly
trivial) the adjoint groups of the Mumford-Tate groups for
$\Phi$. Notice that by definition the group $\LL_Y$ surjects onto $\LL_S$.
The corresponding Mumford-Tate groups for $\Phi'$ are then $\HH_S^\ad$
for $S$, and $\HH_Y^\ad \times \ker(\LL_Y \twoheadrightarrow \LL_S)$
for $Y$. In particular, although the Hodge codimensions $\Hcd(S,
\Phi)$ and $\Hcd(S, \Phi')$ (resp. $\Hcd(Y, \Phi)$ and $\Hcd(Y,
\Phi'))$ differ in general, one still has
the equality:
$$ \Hcd(S, \Phi) - \Hcd(Y, \Phi) = \Hcd(S, \Phi')- \Hcd(Y,
\Phi')\;\;.$$
Hence the result.
\end{rmk}

\subsection{Comparison with \cite{klin}}\label{newzpsection} 

\medskip
In \cite{klin} the second author defined a smaller Hodge codimension,
which will be called in this paper the {\em horizontal Hodge
  codimension}. Similarly we can also define a {\em horizontal
  monodromic codimension}.

Let $(\G, D)$ be a Hodge datum. Each point $x \in D$
defines a $\QQ$-Hodge-Lie algebra structure on the adjoint Lie algebra $\mathfrak{g}^\ad$ of
$\G$, hence a Hodge filtration $F^p_x \mathfrak{g}^\ad_\CC$ on $\mathfrak{g}^\ad_\CC$ of $\G(\CC)$. 
This Hodge filtration induces a decreasing
filtration on the tangent space $T_xD^{\vee}=\mathfrak{g}^\ad_\CC/ F_x^0
\mathfrak{g}^\ad_\CC$. These filtrations glue together to induces a
$\G(\CC)$- equivariant holomorphic decreasing filtration $F^\bullet
TD^{\vee}$, $\bullet <0$, on the $\G(\CC)$-homogeneous tangent bundle $TD^{\vee}$. The
horizontal tangent bundle $T_hD$ of $D$ is (the restriction to $D$ of) the
$\G(\CC)$-equivariant subbundle $F^{-1} TD^{\vee}$. Griffiths' transversality for VHS, i.e. the fact that
$\nabla F^\bullet \cV\subset F^{\bullet -1}\cV \otimes \Omega^1_S$,
translates into the differential constraint
$(d \Phi) (TS)  \subset T_h(\Gamma \backslash D)$. Notice that the
rank of the holomorphic vector bundle $T_h(\Gamma \backslash D)$
equals the dimension of $\mathfrak{g}_{Y}^{-1}=\mathfrak{g}_{Y}^{-1,1}$.

\begin{defi}[{\cite[Definition 1.5 and 1.7]{klin}}]\label{defatyklin}
Let $\VV$ be a polarizable $\ZZ$VHS on a smooth connected complex quasi-projective variety
$S$ and let $Y \subset
S$ be a closed irreducible algebraic subvariety, with generic
Mumford-Tate datum $(\G_{Y}, D_Y)$ for $\VV$ . The \emph{horizontal
  Hodge codimension} of $Y$ for $\VV$ is 
\begin{displaymath}
\hHcd(Y,\VV):=\dim_\C \left(\mathfrak{g}_{Y}^{-1} \right) -\dim \Phi (Y^{\an})\;.
\end{displaymath}
Similarly the \emph{horizontal
  monodromy codimension} of $Y$ for $\VV$ is 
\begin{displaymath}
\hMcd(Y,\VV):=\dim_\C \left(\mathfrak{h}_{Y}^{-1} \right) -\dim \Phi (Y^{\an})\;.
\end{displaymath}

\end{defi}

\begin{defi}
A closed irreducible algebraic subvariety $Y\subset S$ is said to be \emph{
  horizontally atypical} for $\VV$ if
\begin{displaymath}
\hHcd(Y,\VV)<\hHcd(S,\VV).
\end{displaymath}

It is said to be {\em horizontally monodromically atypical} for $\VV$ if
\begin{displaymath}
\hMcd(Y,\VV)<\hMcd(S,\VV).
\end{displaymath}
\end{defi}

\begin{lem}\label{rmk1zp}
  Let $\VV$ be a polarizable $\ZZ$VHS on a smooth complex quasi-projective
  variety $S$. Then:
  \begin{itemize}
    \item[(a)] For any closed irreducible algebraic subvariety $Y
  \subset S$:
\begin{displaymath}
\hHcd(Y,\VV) \leq \Hcd(Y,\VV)\quad \textnormal{and} \quad \hMcd(Y,\VV) \leq \Mcd(Y,\VV)\;\;;
\end{displaymath}
\item[(b)] $\hHcd(Y,\VV) =\Hcd(Y,\VV)$ if and only if $(\G_Y,D_Y)$ is of Shimura
type;
\item[(c)]
A subvariety which is horizontally atypical (resp. horizontally monodromically atypical)
 is atypical (resp. monodromically atypical).
\end{itemize}
\end{lem}

\begin{proof}

  For $(a)$: The inequality $\hHcd(Y,\VV) \leq \Hcd(Y,\VV)$ follows immediately
  from the obvious inequality $\dim T_h
  D_{\G_{Y}} \leq \dim T D_{\G_{Y}}$. Similarly $\hMcd(Y,\VV) \leq \Mcd(Y,\VV)$ follows immediately
  from $\dim T_h
  D_{H_{Y}} \leq \dim T D_{H_{Y}}$.

  For $(b)$: The equality $\hHcd(Y,\VV) = \Hcd(Y,\VV)$ implies that $F^1
  \mathfrak{g}_{Y,\CC} = \mathfrak{g}_{Y, \CC} $, i.e. $(\G_Y,D_Y)$ is
    of Shimura type.

For $(c)$: Suppose hat $Y \subset S$ is a closed irreducible algebraic
subvariety which is horizontally atypical for $\VV$, that is
$\hHcd(Y, \VV) <\hHcd(S, \VV)$. As
\begin{displaymath}
\Hcd(Y, \VV) = \hHcd(Y, \VV) +
\sum_{k\geq 2} \fg^{-k}_Y
\end{displaymath}
 and $\Hcd(S, \VV) = \hHcd(S, \VV) +
\sum_{k\geq 2} \fg^{-k}_S$, the inequality 
$\Hcd(Y,\VV)<\Hcd(S,\VV)$ follows immediately from the inclusions
$\mathfrak{g}_Y^{-k} \subset \mathfrak{g}_S^{-k}$
for all $k\geq 2$. Thus $Y$ is atypical for $\VV$.
The assertion concerning monodromically atypicality is proved in the same
way, replacing the Lie algebra of the generic Mumford-Tate group with
the Lie algebra of the algebraic monodromy group.
\end{proof}

\subsection{Zilber--Pink Conjecture (strong form)}

Let us now state the Zilber-Pink conjecture for $\ZZ$VHS. It is an
enhanced version (using atypicality rather than horizontal
atypicality) of the Main Conjecture proposed in \cite{klin} (see Conjectures 1.9, 1.10,
1.11 and 1.11 in \emph{op. cit.}). In view of \Cref{rmk1zp} it implies
the Main Conjecture in {\em loc. cit}.

\begin{conj}\label{newZP}
For any irreducible smooth quasi-projective variety $S$ endowed with a
polarizable variation of Hodge structures $\VV\to S$, the following
equivalent conditions hold true: 
\begin{itemize}
\item[(a)] The subset $\HL(S, \VV^\otimes)_\atyp$ is a finite union
  of maximal atypical special subvarieties of $S$;
\item[(b)] The subset $\HL(S, \VV^\otimes)_\atyp$ is a strict algebraic subvariety of $S$;
\item[(c)] The subset $\HL(S, \VV^\otimes)_\atyp$ is not Zariski-dense in $S$;
\item[(d)] The variety $S$ contains only finitely many irreducible subvarieties optimal for $\VV$.
\end{itemize}
\end{conj}
The equivalence between the four conditions presented above is obtained by arguing as in \cite[Proposition 5.1]{klin}.

\begin{rmk}
Strictly speaking a Zilber--Pink type conjecture predicts only the
behaviour of atypical intersections. It should be implicitly understood that every other
intersection (namely a typical intersection) has a more predictable
behaviour and satisfies an all or nothing principle: this is the content
of \Cref{conj-typical}. Once this is
understood the difference between \Cref{newZP} and the Main Conjecture
of \cite{klin} becomes crucial, since they would provide two different
definitions of typicality. 
\end{rmk}

\subsection{An example: two-attractors}

In this section we describe a concrete example which is
atypical, and hence covered by \Cref{newZP}, but not
horizontally atypical hence not covered by the weaker \cite[Conjecture
1.9]{klin}. This example already attracted some attention in the
literature, and the link with the Zilber--Pink conjecture could give a
strategy for attacking it.

\medskip
Consider a Calabi-Yau Hodge structure $V$ of weight $3$ with Hodge numbers
$h^{3,0}= h^{2,1}=1$ (for instance the Hodge structure given by the mirror dual
quintic). Its universal deformation space is a quasi-projective curve
$S$, since $h^{2,1}=1$ (the projective line minus three
points in the case of the mirror dual
quintic), which carries a $\Z$VHS $\VV$ of the same type. This gives a
non-trivial period
map $ \Phi: S^{\an} \to \Gamma \backslash D$, where $D=
\mathbf{Sp}(4, \RR)/U(1)\times U(1)$. The Hodge variety $\Gamma
\backslash D$ is a complex
manifold of dimension $4$, while its horizontal tangent bundle has
rank $2$ (see for example \cite[page 257-258]{MR2918237}). The
analytic curve $\Phi(S^\an)$ is known to be Hodge generic in $\Gamma \backslash D$ (in fact its algebraic monodromy is
$\mathbf{Sp}_4$). Thus $\hHcd(S, \VV) = 2-1
=1$ while $\Hcd(S, \VV) = 4-1=3$.

\medskip
Consider the special points $s \in S$
where the weight 3 Hodge structure  
\begin{displaymath}
\VV_{s, \C}= V_s^{3,0} \oplus V_s^{2,1} \oplus V_s^{1,2}  \oplus V_s^{0,3}
\end{displaymath}
splits as a sum of two twisted weight one Hodge structures: $(V_s^{2,1}
\oplus V_s^{1,2})$ and its orthogonal for the Hodge metric $ (V_s^{3,0}
\oplus V_s^{0,3})$. Following Moore \cite{moorepaper,
  moorenotes}, such points $s\in S$ are referred to as
\emph{two-attractors}. They are irreducible components of
the intersection of $\Phi(S^\an)$ with a product $Y_n\subset \Gamma\backslash D$ of two modular curves (one of them non-horizontally embedded).

\medskip
If the two-attractor $s\in S$ is moreover a CM-point then $\hHcd(s, \VV) =0-0=
0$, therefore $s$ is horizontally atypical (hence also atypical). On the other hand if the
two-attractor $s\in S$ is not CM then $\hHcd(s, \VV) = 1-0=1$ and
$\Hcd(s, \VV)= 2-0=2$. Hence $s$
is horizontally-typical, but atypical.

\medskip
We remark here that \cite[Conjecture 1.9]{klin} predicts that the subset of CM two-attractor
points in $S$ is finite, but says nothing for the full
two-attractor locus. On the other hand \Cref{newZP} predicts that the full two-attractor locus in $S$ is finite. It agrees with the
expectation of  \cite[(page 44)]{4authors}, where the authors propose
some evidence for the finiteness of the two-attractors on $S$.

\section{Atypical Locus -- The Geometric Zilber--Pink conjecture}\label{atypicalsection}
In this section we first prove a version of \Cref{geometricZP} for the weakly 
atypical weakly special locus $\HL(S, \VV^{\otimes})_{\ws, \watyp}$:
the maximal monodromically atypical weakly special subvarieties of positive
period dimension arise in a finite number of families whose geometry
we control. The proof is mainly inspired by the arguments
appearing in \cite[Theorem 4.1]{MR3177266}, \cite[Theorem
1.2.1]{bu} and the work of Daw and Ren \cite{MR3867286}. In \Cref{proofGeometricZP} we explain how the
proof adapts to the case of $\HL(S, \VV^\otimes)_{\pos, \atyp}$, thus 
proving \Cref{geometricZP}.

\begin{theor}[Geometric Zilber--Pink conjecture for the weakly special atypical locus]\label{geometricZPweakly}
Let $\VV$ be a polarizable $\ZZ$VHS on a smooth connected complex quasi-projective variety
$S$, with generic Hodge datum $(\G, D)$.
Let $Z$ be an irreducible component of the Zariski closure
$\overline{\HL(S, \VV^{\otimes})_{\ws, \watyp}}^\Zar$ in $S$. Then:

\begin{itemize}
  \item[(a)] Either $Z$ is a maximal monodromically atypical special subvariety;
    \item[(b)] Or the adjoint Mumford-Tate group $\G_Z^\ad$
  decomposes as a non-trivial product $\HH^\ad_Z \times \LL_Z$; $Z$ contains a Zariski-dense
set of fibers of $\Phi_{\LL_{Z}}$ which are monodromically atypical weakly special
subvarieties of $S$ for $\Phi$, where (possibly up to an \'{e}tale covering) $$
\Phi_{|Z^\an}= (\Phi_{\HH_{Z}}, \Phi_{\LL_{Z}}): Z^\an \to  \Gamma_{\G_{Z}}\backslash D_{G_{Z}}= \Gamma_{\HH_{Z}}
\backslash D_{H_{Z}} \times  \Gamma_{\LL_{Z}}\backslash D_{L_{Z}} \subset \Gamma \backslash
D
\;\;;$$
and $Z$ is Hodge generic in a special subvariety $\Phi^{-1}(
\Gamma_{\G_{Z}}\backslash D_{G_{Z}})^0$ of $S$ for $\Phi$ which is monodromically typical (and therefore typical). 
\end{itemize}

\end{theor}

\begin{rmk}
\Cref{geometricZPweakly} is more satisfying than \Cref{geometricZP}:
while the second branch of the alternative in \Cref{geometricZP}
should not occur due to \Cref{main conj},
\Cref{geometricZPweakly} is a complete description of varieties
containing a Zariski-dense set of monodromically atypical weakly
special subvarieties. 
\end{rmk}

\subsection{Preliminaries for the proof of \Cref{geometricZPweakly}} 

\subsubsection{Definable fundamental sets}\label{fundset}
Let $\Phi: S  \to \Gamma_\HH \backslash D_H$ be the period map for $\VV$.
We refer the reader to \cite{MR1633348} for the notion of o-minimal
structure and to \cite{BKT} for its use in Hodge theory. As recalled in \Cref{zassection},
the Mumford-Tate domain $D_H$ is a real semi-algebraic open subset of
its compact dual $D_H^\vee$ and that its 
\emph{algebraic subvarieties} are the irreducible complex analytic
components of the intersections with $D_H$ or complex algebraic
subvarieties of $D_H^\vee$. Recall also that $S$ being an algebraic
variety, it is naturally definable in any extension of
the o-minimal structure $\R_{\operatorname{alg}}$. {\it From now on,
definable will be always understood in the o-minimal structure
$\R_{\an, \exp}$.}  

\medskip
Let us introduce a notion of ``definable fundamental set''  of $S$ for $\Phi$, arguing as at the beginning of
\cite[Section 3]{MR3958791}. Let $(\overline{S},E)$ be a log-smooth
compactification of $S$, and choose a definable atlas of $S$ by
finitely many polydisks $\Delta^k \times (\Delta ^*)^\ell$. Let 
\begin{displaymath}
\exp : \Delta^k \times \Hh^\ell \to \Delta^k \times (\Delta ^*)^\ell
\end{displaymath}
be the standard universal cover, and choose 
\begin{displaymath}
\Sigma = [-b,b] \times
[1, + \infty[
\subset \Hh
\end{displaymath}
such that $\Delta^k \times \Sigma ^\ell$ is a fundamental
set for the $\ZZ^{\ell}$-action by covering transformation. Let $\F$ be the disjoint union of
$\Delta^k \times \Sigma^\ell $ over all charts and choose lifts
$\Delta^k \times \Hh^\ell  \to D$ of the period map restricted to each
chart to obtain a lift $\tilde{\Phi}=\tilde{\Phi}_\F : \F
\to D_H$. From the Nilpotent Orbit Theorem \cite[(4.12)]{sch73} we
obtain the following diagram in the category of definable complex
manifolds: 

\begin{center}
\begin{tikzpicture}[scale=2]
\node (A) at (-1,1) {$\F$};
\node (B) at (0,1) {$D_H$};
\node (C) at (-1,0) {$S^{\an}$};
\path[->,font=\scriptsize,>=angle 90]
(A) edge node[above]{$\tilde{\Phi}$} (B)
(A) edge node[right]{$\exp$} (C);
\end{tikzpicture}.
\end{center}

\subsubsection{Notation for the proof of \Cref{geometricZPweakly}}
Let $(\MM, D_M)$ be a weak Hodge subdatum of $(\HH, D_H)$, and set $M=
\MM(\R)^+ \subset H=\HH(\R)^+$. We recall here that $M$ and $H$ are
semialgebraic, that is definable in $\R_{\alg}$. 
We introduce the set
\begin{equation} \label{defiPi}
\Pi(\MM, D_M):=\{(x,g) \in \F \times H\; |\;  \tilde{\Phi}(x) \in
g\cdot D_M\}\;\;,
\end{equation}
where $\F$ denotes the definable fundamental set constructed in
\Cref{fundset}. Notice that if $(x, g) \in \Pi(\MM, D_M)$ then $g \cdot M
g^{-1}\tilde{\Phi}(x) = g \cdot D_M$ is, following the conventions from
\Cref{zassection}, an algebraic subvariety of $D_H$. Denote by 
\begin{displaymath}
\pi : D_H \to \Gamma_{\HH}\backslash D_H
\end{displaymath}
the projection map. We have that $\pi (gMg^{-1}\cdot \tilde{\Phi}(x)) \subset \Gamma_\HH
\backslash D_H$ is closed in $\Gamma_\HH \backslash D_H$ only if it  
defines a weakly special subvariety of $\Gamma \backslash D$.

By ``dimension'' we will always mean the complex
dimension (and possibly the local dimension at some point). Consider the function
$$
d : \Pi(\MM, D_M) \to \RR, \qquad (x,g) \mapsto d(x, g):=
\dim_{\tilde{\Phi}(x)} \left(g \cdot D_M \cap \tilde{\Phi}(\F)\right). 
$$
It defines a natural decreasing filtration (for $0 \leq j <n := \dim \Phi(S^\an)$)
\begin{equation}
\Pi^{j}(\MM, D_M):=\{(x,g) \in \Pi(\MM, D_M) : d(x,g) \geq j\} \subset \Pi(\MM, D_M).
\end{equation}

Finally we define
\begin{equation}\label{sigmadef}
\Sigma^j(\MM, D_M): = \{ (g \MM g^{-1}, g\cdot  D_M)\;|\; \exists (x, g) \in
  \Pi^{j}(\MM, D_M)\}\;\;.
\end{equation}
Notice that if we write 
\begin{equation}\label{taumap}
\tau: \Pi^j(\MM, D_M) \to H/N_{H}(M), \ \ \ (x,g) \mapsto g \cdot N_{H}(M),
\end{equation}
where $N_{H}(M)$ denotes the normaliser of $M=\MM(\R)^+$ in $H$, then one also has
\begin{equation*}
\Sigma^j(\MM, D_M) = \tau (\Pi^j(\MM, D_M))\;\;.
\end{equation*}

\begin{lem}
  Let $(\MM, D_M) \subset  (\HH, D_H)$ be a weak Hodge subdatum. Then
  the subsets $\Pi(\MM, D_M) \subset \F \times H$, $\Pi^{j}(\MM,
  D_M) \subset \F \times H$, and $\Sigma^j(\MM, D_M) \subset
  H/N_{H}(M)$ are definable subsets (where $\F \times H$ and
  $H/N_{H}(M)$ are endowed with their natural definable structures).
\end{lem}

\begin{proof}
  That $\Pi(\MM, D_M) \subset \F \times H$ is a definable
  subset follows from (\ref{defiPi}), and from the facts that $\tilde{\Phi}: \F \to D_H$ is
  definable (see \Cref{fundset}), that $D_M \subset D_H$ is definable
  and that the $H$-action on $D_H$ is definable.
  
Observe that, since the dimension of a
definable set at a point is a definable function, the set
$\Pi^j(\MM, D_M)$ is a definable subset of $\F \times H$.

Finally, as the projection $H \to H/N_{H}(M)$ is semi-algebraic, the map
$\tau$ introduced in (\ref{taumap}) is definable. Hence its image
$\Sigma^j(\MM, D_{\MM}) \subset H/N_{H}(M)$ is a definable subset.   
\end{proof}

\subsection{Main Proposition and Ax-Schanuel}\label{mainpropsec}

From now on, we may and do assume that the period map $\Phi$ is
immersive. This is possible thanks to \cite[Theorem 1.1]{gaga}, which
asserts that the image of the period map is algebraic. We do this because the atypicality condition appearing in
the Ax-Schanuel theorem concerns the dimension of $S$, rather than its
period dimension (as it appears in our notion of atypical special
subvarieties \Cref{typatyp}).

\begin{defi}  \label{atypical mono}
Let $\VV$ be a polarizable $\ZZ$VHS on a smooth connected complex quasi-projective variety
$S$, with generic Hodge datum $(\G, D)$ and generic weak Hodge datum
$(\HH, D_H)$. A weak Hodge subdatum $(\MM, D_M)$ of $(\G, D)$ is said to be
\emph{monodromically atypical} for $\VV$ if there exists $Z\subset S$ a monodromically atypical
weakly special subvariety of $S$ for $\VV$, such that $(\HH_Z, D_{H_{Z}})= (\MM, D_M)$.
\end{defi}

\begin{rmk}
Let $Z \subset S$ be as in \Cref{atypical mono}. Thus
 $Z =\Phi^{-1}(\Gamma_{\MM} \backslash D_{M})^0$ and either $Z$ is
 singular for $\VV$ or
\begin{equation}\label{eqatyz}
\dim \Phi(S^\an) - \dim \Phi(Z^{\an}) < \dim D_{H} - \dim D_{M}\;\;.
\end{equation}
\end{rmk}

Recall that, there are only finitely many conjugacy classes of semisimple subgroups of $H$. In particular, there are only
finitely many weak Hodge subdata $(\HH_i, D_{H_{i}})$, $ i \in \{1, \cdots, w\}$, of
$(\G, D)$ such that every monodromically atypical weakly special subvariety of $S$ for $\mathbb{V}$ is associated to a $G$-conjugate of some $(\HH_i, D_{H_{i}})$. Let $m$ be the maximal period dimension $\dim
  \Phi(Z^{\an})$, for $Z \subset S$ ranging through the monodromically atypical weakly special
  subvarieties of $S$ for $\VV$, which are of positive period dimension
  and non-singular for $\VV$ (by assumption $m >0$).

The main ingredient in the proof of \Cref{geometricZPweakly} is the
following proposition:

\begin{prop}\label{mainpropwithas}
  Let $(\MM, D_M) \subset (\G, D)$ be a weak Hodge subdatum, monodromically atypical
for $\VV$ giving rise to a $m$-dimensional monodromically atypical weakly special
  subvariety of $S$. Then the set $\Sigma^m(\MM, D_M)$ is finite. 
\end{prop}

The proof uses (multiple times) the Ax--Schanuel Theorem recalled in \Cref{zassection}.

\begin{proof}
As $\Sigma^m(\MM, D_M)$ is a definable subset of $H/N_{H}(M)$, it is
enough to show that $\Sigma^m(\MM, D_M)$ is countable to conclude
that it is finite.

To show that
$\Sigma^m(\MM, D_M)$ is countable it is enough to show that for each
$(x,g)\in \Pi^m(\MM, D_M)$ the subset $$\pi (g\cdot  D_M)\subset \Gamma_\HH
\backslash D_H \subset \Gamma \backslash D$$ is a weakly special subvariety (as the Hodge variety
$\Gamma \backslash D$ admits only countably many families of weakly
special subvarieties).

For $(x,g)\in \Pi^m(\MM, D_M)$ consider the algebraic subset of $S\times {D_H}$
\begin{displaymath}
W_{(x,g)}:= S \times g\cdot  D_M\;\;.
\end{displaymath}
Choose $U_{(x,g)}$ an irreducible complex analytic component of
maximal dimension at $(\exp(x),\tilde{\Phi}(x))$ of
\begin{displaymath}
W_{(x,g)} \cap \left (S \times_{ \Gamma_\HH \backslash D_H} D_H \right )\subset S\times D_H.
\end{displaymath}

\noi
We claim that $U_{(x,g)}$ is an atypical intersection, that is
\begin{equation*}
\codim_{S \times D_H}U_{(x,g)} <  \codim_{S \times D_H}W_{(x,g)}
+ \codim_{S \times D_H}  \left(S \times_{ \Gamma_\HH \backslash D_{H}} D_H
\right).
\end{equation*}
Indeed this last inequality can be rewritten as
\begin{equation}\label{equ}
\dim \Phi(S^\an) - \dim  U_{(x,g)} < \dim D_H - \dim g \cdot D_M\;\;.
\end{equation}
As $(x, g) \in \Pi^m(\MM, D_M)$, one has $\dim  U_{(x,g)} =m = \dim \Phi(Z^\an)$ . On the
other hand $$\dim g \cdot D_M= \dim
D_M\;\;.$$ 
Hence the required inequality (\ref{equ}) is exactly the inequality
(\ref{eqatyz}), thus $
U_{(x,g)}$ is an atypical intersection.

The Ax--Schanuel conjecture for $\Z$VHS,
\Cref{astheorem}, thus asserts that the projection $S_{(x, g)}$ of $U_{(x,g)}$ to $S$ is
contained in some strict weakly special subvariety of $S$. Let $S'= \Phi^{-1}
(\Gamma_{\HH'} \backslash D_{H'})$ be
such a strict weakly special subvariety of $S$
containing $S_{(x,g)}$, of smallest possible period dimension:
\begin{displaymath}
S_{(x, g)} \subset S' \subsetneq S.
\end{displaymath}
In particular $\dim \Phi(S'^{\an}) \geq \dim S_{(x, g)} =m$ with equality if and
only if $S'= S_{(x,g)}$.

\medskip
Suppose that $S'$ is monodromically typical for $\VV$. Thus  
\begin{equation} \label{equation2}
\dim \Phi(S^{\an}) - \dim \Phi({S'}^\an)= \dim D_{H}- \dim D_{H'} \;\;.
\end{equation}
Consider the algebraic subvariety $$W'_{(x, g)}: = S' \times g \cdot
D_M\subset S' \times D_{H'}$$ and let $U'_{(x, g)}$ 
be an irreducible complex analytic component at $(\exp(x),\tilde{\Phi}(x))$ of 
\begin{displaymath}
  W'_{(x, g)} \cap (S'\times_{\Gamma_{\HH'} \backslash D_{H'}} D_{H'})
    \subset 
    S' \times D_{H'}\;\;
  \end{displaymath}
containing $U_{(x,g)}$.
The typicality condition (\ref{equation2}) ensures that $U'_{(x, g)}$ is
still an atypical intersection. Applying again the Ax--Schanuel
conjecture, \Cref{astheorem} produces a strict weakly special
subvariety $S'' \subsetneq S'$ containing $S_{(x, g)}$. This contradicts the minimality of $S'$.

\medskip
Thus $S'$ has to be monodromically atypical (and non-singular) for $\VV$. As
$\dim \Phi(S') \geq m$ it follows from the maximality of $m$ that $\dim
\Phi(S')=m$, thus 
$S' =S_{(x, g)}$.
This is to say that $\pi (g \cdot D_M)$ is a weakly special
subvariety of $\Gamma \backslash D$.

\medskip
This concludes the proof of \Cref{mainpropwithas}.
\end{proof}

\subsection{Proof of \Cref{geometricZPweakly}}	\label{section64}
We have all the elements to finally prove \Cref{geometricZPweakly}. For some similar arguments we refer to the
proofs of \cite[Theorem 4.1]{MR3177266} and \cite[Theorem 1.2.1. and
Section 6.1.3]{bu}.

\begin{proof}[Proof of \Cref{geometricZPweakly}]

  We need the following lemma, describing the weakly special
  subvarieties of $S$ singular for $\VV$:
  
\begin{lem} \label{below}
Let $Y\subset S$ be a weakly special subvariety of $S$
for $\VV$, singular for $\VV$, and maximal for these properties. Let $S'$ be an irreducible component of
$S^\sing_\VV$ containing $Y$. If $Y$ satisfies the equality
\begin{equation} \label{egalite}
  \dim \Phi({S}^\an) - \dim \Phi(Y^{\an}) = \dim
  \Gamma_{\HH} \backslash D_{H} - \dim \Gamma_{\HH_{Y}}\backslash
    D_{H_{Y}}\;\;,
  \end{equation}
then either $Y \subset S'$ is monodromically atypical for
$\VV_{|S'}$ or $Y=S'$.
\end{lem}

\begin{proof}
 Suppose that $Y \subset S'$ is monodromically typical for
 $\VV_{|S'}$. In particular:
 \begin{equation*} 
  \dim \Phi({S'}^\an) - \dim \Phi(Y^{\an}) = \dim
  \Gamma_{\HH_{S'}}\backslash D_{H_{S'}} - \dim \Gamma_{\HH_{Y}}\backslash
    D_{H_{Y}}\;\;.
  \end{equation*}
  Subtracting from (\ref{egalite}) and as $\dim \Phi({S'}^\an)  <
  \dim\Phi(S^\an)$ by definition of the singular locus, it follows
  that $S'$ is contained in a strict weakly special subvariety of
  $S$. As $S'$ contains $Y$ which is maximal among the weakly special
  subvarieties of $S$, we deduce by irreducibility of $Y$ and $S'$
  that $S'=Y$.
\end{proof}

Let $Z$ be as in \Cref{geometricZPweakly}. If $Z$ contains a
Zariski-dense set of weakly special subvarieties of $S$ singular for $\VV$, each satisfying the equality~(\ref{egalite}), 
and maximal for these properties, it then follows from
\Cref{below} that:

- Either $Z$ coincides with an irreducible component
$S'$ of $S^\sing_\VV$ which is a weakly special subvariety of $S$ for
$\VV$. But then $Z$ satisfies $(a)$ of \Cref{geometricZPweakly};

- Or $Z$ is an irreducible component of $\overline{\HL(S',
  \VV_{|S'}^\otimes)_{\ws, \watyp}}^\Zar$ for such an irreducible component
$S'$.

\medskip
Without loss of generality (replacing $S$ by $S'$ if necessary, and
iterating) we are reduced to the case
where $Z$ contains a Zariski-dense set of monodromically atypical
weakly special subvarieties of $S$ for $\VV$, non-singular for $\VV$
and of positive period dimension.

\medskip
Recall that $m$ is the maximum of the period dimensions of the weakly
atypical weakly special subvarieties of $S$. For any $j\leq m$
consider the set  
\begin{multline}
  \cE_j := \{ \textnormal{maximal monodromically atypical weakly special subvarieties
    of}\; S \; \\
  \textnormal{ non-singular for $\VV$ and of period dimension}\; j\}\;\;.
  \end{multline}
and, following the notation introduced in \Cref{mainpropsec}, write
\begin{equation} \label{Ej}
\mathcal{E}_j=\coprod _{1\leq i \leq w} \;\coprod_{(\MM, D_M) \in
  \Sigma^j(\HH_i, D_{H_{i}})} \mathcal{E}_j(\MM, D_{M})\;\;,
\end{equation}
where  $\mathcal{E}_j(\MM, D_{M})$ denotes the set of
atypical weakly special subvarieties of $S$ of period dimension $j$
with generic weak Hodge datum $(\MM, D_M)$.

\medskip
Let $Z$ be as in \Cref{geometricZPweakly}. There exists a larger $j \in
\{1, \cdots, m\}$ and an $i \in \{1,
\cdots, w\}$ such that the union of the special subvarieties of
$\cE_j(\MM, D_M)$ contained in $Z$, for $(\MM, D_M)$ ranging through
$\Sigma^j(\HH_i, D_{H_{i}})$, is Zariski-dense in $Z$.

\medskip
Suppose first that $j=m$. As $\Sigma^m(\HH_i, D_{H_{i}})$ is a
finite set by \Cref{mainpropwithas}, there exists $(\MM, D_M) \in \Sigma^m(\HH_i,
D_{H_{i}})$ such that the union of the special subvarieties of
$\cE_j(\MM, D_M)$ contained in $Z$ is Zariski-dense in $Z$.
Let $\LL:=\mathbf{Z}_{\G^\ad}(\MM^\ad)$.

\medskip
Either $L$ is compact, in
which case, similarly to \cite[Proof of Theorem 4.1]{MR3177266}, we observe that
$\mathcal{E}_m(\MM,D_{M})$ is a finite set as 
$\Sigma^m(\MM, D_M)$ is finite, and that each element in
$\mathcal{E}_m(\MM,D_{M})$ is a special subvariety. It follows that
$Z$ satisfies (a) in \Cref{geometricZPweakly}.

\medskip
Or $L$ is not compact. In that case the 
pair $$(\MM', D_{M'}):=(\MM^\ad \cdot \LL, Z_{G}(M) \cdot D_M)$$ is a Hodge subdatum of $(\G^\ad,
D)$, satisfying the inclusions of weak Hodge data
\begin{displaymath}
(\MM,D_M)\subsetneq (\MM', D_{M'}) \subset (\mathbf{G^\ad},D)\;\;.
\end{displaymath}
Let $Z'\subset S$ be an irreducible component of
$\Phi^{-1} (\Gamma_{\MM'} \backslash D_{M'})$ containing $Z$. It is
a special subvariety of $S$ containing $Z$. By maximality of $Z$ among
the monodromically atypical weakly special subvarieties, we observe that $Z'$ is weakly
typical and of product type. Thus $Z$ satisfies (b) in \Cref{geometricZPweakly}.

This finishes the proof of \Cref{geometricZPweakly} in the case where $j=m$.

\medskip
We now argue by induction (downward) on $j$. Consider
the set 
\begin{displaymath}
\Sigma_2^j(\HH_i,D_{H_{i}}):= \{(g\HH_ig^{-1}, g \cdot D_{H_{i}})\;| \; \exists (x,g)\in
\Sigma^j (\HH_i,D_{H_{i}}), \;\; \forall j'>j \;\; S_{x,g}\notin
\mathcal{E}_{j'} \},
\end{displaymath}
where $\Sigma ^j(\HH_i,D_{H_{i}})$ is defined as in (\ref{sigmadef}). The
inductive assumption implies that the second condition is
definable. The proof of \Cref{mainpropwithas} shows that
$\Sigma_2^j(\HH_i,D_{H_{i}})$ is finite. Arguing as in the case $j=m$
we conclude again that $Z$ satisfies either (a) or (b) of
\Cref{geometricZPweakly}.

This concludes the proof of the Geometric Zilber--Pink
conjecture for monodromically atypical weakly special subvarieties. 
\end{proof}

\subsection{Proof of \Cref{geometricZP}} \label{proofGeometricZP}
To prove \Cref{geometricZP} one repeats the proof of 
\Cref{geometricZPweakly}, replacing the inputs ``weakly special''
and ``monodromically atypical'' by ``special'' and ``atypical'' respectively; and, in the proof,
``weak Hodge subdatum'' by ``Hodge subdatum'' and all the inequalities
concerning the monodromic codimension by one with the Hodge
codimension. As atypical subvarieties are monodromically atypical by 
\Cref{HatypisMatyp}, we can still apply the Ax-Schanuel conjecture in
the proof of \Cref{mainpropwithas} to produce a weakly special
subvariety $S'$, which is not only monodromically atypical but atypical. The
argument then gives rise to the alternative (a) or (b) in
\Cref{geometricZP} in the same way it gave rise to the alternative (a)
or (b) for \Cref{geometricZPweakly} (for the branch (a) we use
\Cref{converse} to deduce that $Z$, which is an atypical weakly
special subvariety containing a special subvariety, is actually an
atypical special subvariety).

\section{A criterion for the typical Hodge locus to be empty} \label{criterionsection}

In this section we prove \Cref{level criterion} and \Cref{level22}. Let $\VV$ be a polarizable $\ZZ$VHS
on a smooth connected complex quasi-projective variety 
$S$, with generic Hodge datum $(\G, D)$ and algebraic monodromy
group $\HH$. \Cref{level criterion} and \Cref{level22} are variations
of the following results (see \Cref{addendum Chris} for this deduction):

\begin{theor} \label{level criterion 1}
Let $\VV$ be a polarizable $\ZZ$VHS
on a smooth connected quasi-projective variety 
$S$, with generic Hodge datum $(\G, D)$ and algebraic monodromy group $\HH$. 
If the level of $\VV$ is at least $3$,
then $\VV$ does not admit any strict monodromically typical weakly special subvariety.
\end{theor}

\begin{prop} \label{level criterion 12}
Let $\VV$ be a polarizable $\ZZ$VHS
on a smooth connected quasi-projective variety 
$S$, with generic Hodge datum $(\G, D)$ and algebraic monodromy group $\HH$. 

If the level of $\VV$ is $2$ and  $\fh$ is $\Q$-simple, 
then any strict monodromically typical weakly special subvariety is 
associated to a sub-datum $(\G',D_{\G'})$, such that $\G'^{\ad}$ is simple.
\end{prop}

\subsection{Proof of \Cref{level criterion 1}}\label{firstsectionontyp}

\begin{proof}[\unskip\nopunct]
The proof uses the fact that the $\QQ$-Hodge-Lie algebra $\fh$ of the
algebraic monodromy group $\HH$ is generated in level $1$, see below; and the
analysis by Kostant \cite{Kostant} of root systems of Levi factors for complex
semi-simple Lie algebras.

\begin{defi}
  Let $K=\QQ$ or $\RR$.
  A $K$-Hodge-Lie algebra $\fg$ is said to be \emph{generated in level} $1$
  if the smallest $K$-Hodge-Lie subalgebra of $\fg$ whose
  complexification contains $\fg^{-1}$ and $\fg^1$ coincides with
  $\fg$.
\end{defi}
 \begin{prop} \label{generationlevel1}
Let $\VV$ be a polarizable $\ZZ$VHS
on a smooth connected complex quasi-projective variety 
$S$, with algebraic monodromy group $\HH$. Then the $\QQ$-Hodge-Lie
subalgebra structure $\fh_s$ defined by any point $s \in S$ on $\fh$
is generated in level~$1$.

More precisely, the $\RR$-Hodge-Lie
subalgebra structure defined by any point $s \in S$ on the 
non-compact part $\fh_{\RR}^{\nc}$ of $\fh_\RR:= \fh \otimes_\QQ \RR$ is
generated in degree~$1$, and the compact part $\fh_\RR^{\textnormal{c}}$ is of 
Hodge type~$(0,0)$.
\end{prop}
  
\begin{prop} \label{nosubalgebra}
  Suppose $\fg_\RR$ is a simple $\RR$-Hodge-Lie algebra generated in level
  $1$, and of level at least $3$. If $\fg'_\RR \subset \fg_\RR$ is an $\RR$-Hodge-Lie subalgebra
  satisfying ${\fg'}^i = \fg^i$ for all $|i|\geq 2$ then $\fg'=
  \fg$.

\end{prop}

Assuming for the moment 
\Cref{generationlevel1} and \Cref{nosubalgebra}, let us finish the
proof of \Cref{level criterion 1}. Let $Y\subset
S$ be a monodromically typical weakly special subvariety for $\VV$, with
algebraic monodromy group $\HH_Y$. In the following we fix a point
$x \in D_{H_{Y}} \cap \widetilde{\Phi}(\widetilde{S^\an})$ and consider the $\QQ$-Hodge-Lie algebra pair $\fh_Y
\subset \fh$ defined by this point. The typicality condition writes:
\begin{equation} \label{e1}
  \dim \left ( \bigoplus_{i \geq 1} \fh_Y^{-i}\right ) - \dim
\Phi(Y^{\an}) = \dim  \left ( \bigoplus_{i \geq 1} \fh^{-i} \right ) -
\dim \Phi(S^{\an}) \;\;.
\end{equation}
As $Y$ is monodromically typical, it is also horizontally monodromically typical by \Cref{rmk1zp}(c):
\begin{equation} \label{e1'}\dim \fh_Y^{-1} - \dim \Phi(Y^{\an}) = \dim \fh^{-1} - \dim
  \Phi(S^{\an})\;\;.
  \end{equation}
We thus deduce from (\ref{e1}) and (\ref{e1'}) that $$\dim \left ( \bigoplus_{i \geq
    2} \fh_Y^{-i}\right ) = \dim \left ( \bigoplus_{i \geq 2}
  \fh^{-i}\right )\;\;,$$ thus $\fh_Y^{-i} = \fh^{-i}$ as $\fh_Y^{-i}
\subset \fh^{-i}$, for all $i \geq 2$. As $\fh_Y^i =
\overline{\fh_Y^{-i}}$ and $\fh^i = \overline{\fh^{-i}}$ we finally
deduce:
\begin{equation} \label{e2}
  \fh_Y^i = \fh^i \quad \textnormal{for all  } \, |i| \geq 2,
\end{equation}

The assumption that $\fh$ is of level at least~$3$ says that each
simple $\QQ$-factor of $\fh$ is at level at least~$3$. Moreover the
equality~(\ref{e2}) remains true if we replace $\fh$ by any of its
$\QQ$-simple factors and $\fh_Y$ by its intersection with this simple
factor. To deduce that
$\fh_Y=\fh$, we can thus without loss of generality assume that $\fh$ is
$\QQ$-simple.

By \Cref{generationlevel1} it follows that one simple $\RR$-factor
$\fl_\RR$ of $\fh_\RR^{\nc}$ is generated in level~$1$ and of level at
least~$3$. It then follows from \Cref{nosubalgebra} that the intersection
$\fh_{Y, \RR} \cap \fl_\RR$ equals $\fl_\RR$. As both $\fh_Y$ and
$\fh$ are defined over $\QQ$ and $\fh$ is the smallest $\QQ$-algebra
whose $\RR$-extension contains $\fl_\RR$, we conclude that
$\fh_Y=\fh$. Thus $Y= S$, which finishes the
proof of \Cref{level criterion 1}, assuming \Cref{generationlevel1}
and \Cref{nosubalgebra}.
\end{proof}

\begin{proof}[Proof of \Cref{generationlevel1}]

  Passing to a finite \'etale cover of $S$ if necessary, we can and
  will assume without loss of generality that the monodromy of the
  local system $\VV$ is contained in $\HH(\RR)$.
  Let $\VV_{\fh}$ be the $\QQ$VHS associated to the adjoint
  representation of the algebraic monodromy group $\HH$.
  For each $s \in S$, let $\fh_s$ be the fiber at $s$ of 
  $\VV_{\fh}$. This is a semi-simple $\QQ$-Hodge-Lie
  algebra. We write $\fh_{\RR,s}= \fh_{\RR,s}^{\nc} \oplus
  \fh_{\RR,s}^{\textnormal{c}}$ for the decomposition of the
  $\RR$-Hodge-Lie-algebra $\fh_{\RR,s}$ into its non-compact and
  compact part.

  \medskip
  As $\HH$ has no $\QQ$-anisotropic factor, $\fh_{\RR,s}$ is the smallest subalgebra of $\fh_{\RR, s}$ defined
  over $\QQ$ and containing $\fh_{\RR,s}^{\nc}$. Proving that
  $\fh_{\RR,s}^{\nc}$ is generated in level~$1$ thus implies that $\fh_s$
  is generated in level~$1$. Let us now turn to the proof of this last statement.

  \medskip
 Let $\fh'_{\CC, s} \subset \fh_{\CC,s}$ be the complex
 Lie subalgebra generated by $\fh^{-1}_s$ and $\fh_s^1$. It is naturally defined
  over $\R$: $\fh'_{\CC,s} = \fh'_{\RR,s} \otimes_{\RR} \CC$ and
  $\fh'_{\RR,s} \subset \fh_{\RR,s}$ is a real Lie subalgebra.
 We claim that the collection $\left (\fh'_{\RR,s } \right )_{s \in S}$ defines a local
 subsystem of $\VV_{\fh_{\R}}:= \VV_{\fh} \otimes_\QQ
 \RR$. To prove the claim, it is enough to show that the corresponding holomorphic subbundle of
 $\mathcal{V}_{\fh}:= \VV_{\fh_{\CC}} \otimes_\CC \mathcal{O}_S$ is stable under the flat
 connection $\nabla$ defined by $\VV_{\fh_{\CC}}$. But a local
 holomorphic vector field $X$ on $S$ identifies under the period map
 with a local section $u $ of $F^{-1}\mathcal{V}_{\fh}$ over $\Phi(S^{\an})$;
 and, under this identification, the derivation $\nabla_X$ at a point
 $s$ is nothing else than $\ad \, u$, which preserves $\fh'_{\CC,s}$
 at each point according to the very definition of the latter. Hence
 the claim.

 \medskip
It follows that $\fh'_{\RR,s}$ is an $\HH_s(\RR)$-submodule of the Lie
algebra $\fh_{\RR,s}$ under the adjoint action, i.e. an ideal of
$\fh_{\RR,s}$. As $\fh_{\RR,s}$ is semi-simple we obtain a decomposition of real Lie algebras, hence of
$\RR$-Hodge-Lie algebras (see \Cref{remfactor}):
$$ \fh_{\RR, s}= \fh'_{\RR, s} \oplus \fl_s\;\; ;$$
and of the associated flag domains $D_H = D_{H'} \times D_L$. The
subvariety $\widetilde{\Phi}(\widetilde{S^\an}) \subset D_H$ is horizontal,
thus tangent to $D_{H'}$ at every point. Hence there exists $d_L \in
D_{L}$ such that $\widetilde{\Phi}(\widetilde{S^\an})\subset D_{H'} \times
\{d_{L}\}$. It follows that the monodromy $(\Ad \rho)(\pi_1(S,
s))$ of the real local system $\VV_{\fh_{\RR}}$, hence also its Zariski-closure $\HH_\RR$, is contained in
$\HH' \times \MM_L$, where $\MM_L$ is the $\RR$-anisotropic stabilizer of
  $d_L$ in $\LL$. Thus $\HH_\RR= \HH'_\RR \times \MM_L$ ; $\HH'_\RR $
is the non-compact part of $\HH_\RR$, which is thus generated in
degree~$1$; and $\LL = \MM_L$ is the compact
part of $\HH_\RR$, of pure type $(0,0)$. Hence the result.
\end{proof}

\begin{rmk}
\Cref{generationlevel1} should be compared with
\cite[Prop. 3.4]{Robles}. There Robles obtains a weaker result for the
more general situation of an horizontal subvariety $Z$ of $D_H$ not necessarily 
coming from an $\RR$VHS on a quasi-projective base (her result is stated for $(\HH,D_H)= 
(\G, D)$ but the proof adapts immediately to the general case; moreover, as
indicated to us by C.Robles the $\Q$s appearing in \cite[Proposition 3.4]{Robles} have to be replaced by $\RR$s).
In her case the group $\HH'_\RR$ is not
necessarily a factor of $\HH_\RR$. Our stronger conclusion (and easier
proof) comes from the ubiquitous use of Deligne's semisimplicity
theorem.
\end{rmk}

\begin{proof}[Proof of \Cref{nosubalgebra}]

 The proof follows from the results of Kostant in \cite{Kostant}. We
 will use his notation in our setting to help the reader. Thus let us
 write $\fm:= \fg^0$, a Levi factor of the
 proper parabolic subalgebra $\fq := F^0 \fg_\CC$. Let $\ft :=
 \textnormal{cent} \, \fm$ be the center of $\fm$ and let $\fs: = [\fm,
 \fm]$. Thus $\fm = \ft \oplus \fs$. Let $\fr$ be the orthogonal complement for the Killing form of
 $\fm$ in $\fg$ so that $\fg_\CC = \fm \oplus \fr$ and $[\fm, \fr]
 \subset \fr$. A nonzero element $\nu \in \ft^*$ is called a
 $\ft$-root if $\fg_\nu \not = 0$, where
 $$\fg_\nu = \{ z \in \fg_\CC, \quad \ad \, x(z) = \nu(x) z, \; \forall
 \, x \in \ft\}\;\;.$$
 We denote by $R \subset \ft^*$ the set of all $\ft$-roots. Following
 \cite[Theorem 0.1]{Kostant}, the root space $\fg_\nu$ is an
 irreducible $\ad \, \fm$-module for any $\nu \in R$, and any irreducible $\fm$-submodule
 of $\fr$ is of this form. Moreover if $\nu, \mu \in R$ and $\nu + \mu
 \in R$ then $[\fg_\nu, \fg_\mu] = \fg_{\nu+\mu}$, see indeed Theorem 2.3 in \emph{op. cit.}. 

Let $R^+\subset R$ be the set of
 positive $\ft$-roots defined in \cite[(page 139)]{Kostant}. Thus the
 unipotent radical $\fn = F^1\fg_\CC$ of $\fq = F^0 \fg_\CC$ coincides
 with $\bigoplus_{\nu \in R^+} \fg_\nu$. Let $T \in \ft$ be the grading element defining the Hodge graduation
 $\fg_\CC = \bigoplus_{i \in \ZZ} \fg^i$, see \cite[Section
 2.2]{Robles}. Thus $\fg_\nu \subset \fg^{\nu(T)}$ and $\nu \in R^+$
 if and only if $\nu(T) >0$.

 A $\ft$-root in $R^+$ is called simple if it cannot be written as a sum
 of two elements of $R^+$. Let $R_\simple \subset R^+$ denote the
 set of simple roots. In \cite[Theorem  2.7]{Kostant} Kostant proves
 that the elements of
$R_\simple$ form a basis of $\ft^*$. If $R_\simple= \{ \beta_1,
\cdots, \beta_{l} \}$ (where $\dim \ft =
 l$) then moreover the $\fg_{\beta_{i}}$, $i \in I:=\{ 1, \cdots, l\}$, generate
 $\fn$ under bracket. 

 \medskip
 As $\fg_\RR$ is a simple real algebra, it follows from \Cref{complex}
 that $\fg_\CC$ is 
 a simple complex Lie algebra. Since $\fg_\RR$ is assumed to be generated
 in level $1$, it follows that for all $i \in I$, $\fg_{\beta_{i}} 
 \subset \fg^1$ (otherwise $\beta_i$ would not be a simple root); and 
 $\fg^1= \bigoplus_{i \in I} \fg_{\beta_{i}}$.

 Let $\fh_\CC \subset \fg_\CC$ be the complex Lie subalgebra generated
 by $\bigoplus_{i \geq 2} \fg^i$ and $\bigoplus_{i \geq 2} \fg^{-i}$.
 It is naturally an $\fm$-submodule of $\fg_\CC$, hence
 $\fh_\CC \cap \fg^{1}$ decomposes as a direct sum
 $\bigoplus_{j \in J\subset I} \fg_{\beta_{j}}$. 
 As $\fg_\RR$ is generated in degree one and the level $k$ of $\fg$ is
 at least $3$, it follows from \cite[Theorem 2.3]{Kostant} that the subspace $[\fg^{-2},
 \fg^3] \subset \fh_\CC \cap \fg^{1}$ is not $0$, thus $J$ is not empty. 
Notice that, for any $i \in I -J$ and $j \in J$,
 one has $[\fg_{\beta_{i}}, \fg_{\beta_{j}}]=0$. Indeed the Killing
 form is $\ad(\mathfrak{m})$-invariant, and $\mathfrak{h}$ is closed
 under $\ad (\mathfrak{m})$; hence $\mathfrak{h}^{\perp}$ is closed
 under $\ad (\mathfrak{m})$, \emph{a fortiori} under
 $\ad(\mathfrak{t})$, whence $(\mathfrak{h}^\perp)^1= \bigoplus_{i\in
   I - J}\fg_{\beta_i}$. Since $\fg=\fh \oplus \fh^{\perp}$ as a Lie
 algebra, the desired conclusion follows.

 If $J \not= I$ this contradicts the
 fact that the restriction to $\ft$ of the highest root of $\fg_\CC$
 is of the form $\sum_{i\in I} n_{\beta_{i}} \beta_i$, with all
 $n_{\beta_{i}}>0$ for all $i \in I$, see \cite[(2.39) and
 (2.44)]{Kostant}. Thus $I=J$ and $\fh_\CC = \fg_\CC$.
\end{proof}

\begin{rmk}[\Cref{nosubalgebra} does not hold in level two]\label{firsrmkweight2}
Let $(\mathbf{G},D)$ be the Hodge datum parametrising polarized Hodge
structures of weight $2$ with Hodge numbers
$(h^{2,0},h^{1,1}, h^{0,2}= h^{2,0})$ for some $h^{2,0}>0$, $h^{1,1} >0$. Let
$(\mathbf{H},D_H)$ be a sub-Hodge datum of $(\mathbf{G},D)$.  
Recall that $\mathfrak{g}^{-2}$ is non zero, as long as
$h^{2,0}>1$. Then the fact that the canonical inclusion 
\begin{displaymath}
\mathfrak{h}^{-2} \subseteq \mathfrak{g}^{-2}
\end{displaymath}
 is an equality implies that $D_H$ is parametrising Hodge structures
 of weight $2$ with Hodge numbers $(h^{2,0}, a, h^{2,0})$
with $0<a < h^{1,1}$. That is, in level two, typical intersections may
(and do) happen, as recalled in \Cref{Picard}. In \Cref{weighttwoprod}, we
prove that the period domain $D_\mathbf{H}$ cannot be a non-trivial
product, indeed in such an example $\HH=\mathbf{SO}(2h^{2,0},a)$ and
$\G=\mathbf{SO}(2h^{2,0},h^{1,1})$. 
\end{rmk}

\subsection{Proof of \Cref{level criterion}} \label{addendum Chris}
\begin{proof}[\unskip\nopunct] 
Without loss of generality (see \Cref{reduction2}) we can assume that $\HH= \G^\der$.
In the proof of \Cref{level criterion 1}, we may then replace the weakly
typical subvarieties by the typical ones, and $\mathfrak{h}_Y$ by
$\mathfrak{g}^\der_Y$. Everything works the same. It is important
that $\HH= \G^\der$ so that \Cref{generationlevel1} applies to
$\fg^\der$. With this formulation, \Cref{level criterion 1} implies
immediately \Cref{level criterion}. 
\end{proof}

\subsection{Proof of \Cref{level criterion 12}}\label{weighttwoprod}

To prove \Cref{level criterion 12}, and therefore \Cref{level22}, it is enough
to prove the following: 
\begin{prop} \label{nosubalgebraprod}
  Suppose $\fg$ is a simple $\QQ$-Hodge-Lie algebra generated in level
  $1$, and of level $2$. If $\fg'  \subset \fg$ is a $\QQ$-Hodge-Lie subalgebra
  satisfying ${\fg'}^{-2} = \fg^{-2}$, then $ \fg'$ is simple. 
\end{prop}

\begin{proof}
Arguing as in the previous section, we can suppose that $\G$ is
$\R$-simple. This implies that $\fg^{-2}$ is a simple (non-zero)
$\fg^0$-module. Notice also that $\fg^0$ is a product $[\fg^{-2},
\fg^2] \times \fu$ (where $\fu$ is possibly trivial). This implies
that $\fg^{-2}$ is of the form $R \otimes S$, in the case where $\fu$
is non trivial, or simply $R$ if $\fu$ is trivial, where $R$ is a simple $[\fg^{-2},
\fg^2]$-module and $S$ a simple $\fu$-module.

\medskip
Let $\fg' \subset \fg$ be a Hodge-Lie subalgebra such
that ${\fg'}^{-2}= \fg^{-2}$. We want to show that $\fg'$ is
simple. Suppose by contradiction it is not simple. Thus
\begin{equation} \label{decompo}
  \fg'= \fg_1' \oplus \fg_2' \;\;,
\end{equation}
for two non-trivial Hodge-Lie subalgebras $\fg_i'$, $i=1,2$.
The decomposition~(\ref{decompo}) gives rise to a decomposition
\begin{equation} \label{decompo2}
  \fg^{\pm 2}= {\fg'}^{\pm 2}= {\fg_1'}^{\pm 2}\oplus {\fg_2'}^{\pm 2} \;\;,
\end{equation}

Suppose first that both ${\fg_1'}^{-2}$ and ${\fg_2'}^{-2}$ are
non-zero. In that case $$[\fg^{-2},
\fg^2] = [{\fg'_1}^{-2}, {\fg'_1}^2] \oplus [{\fg'_2}^{-2},
{\fg'_2}^{-2}]$$ is a non-trivial decomposition, and the decomposition~(\ref{decompo2}) is a
non-trivial decomposition of the $\fg^0$-module $\fg^{-2}$, where
$\fg^0$ acts on ${\fg_i'}^{-2}$, $i=1,2$, through its quotient
$[{\fg'_i}^{-2}, {\fg'_i}^2] \times \fu$. This contradicts
the simplicity of $\fg^{-2}$ as a $\fg^0$-module.

\medskip
Thus, exchanging factors if necessary, we can assume without loss of
generality that ${\fg'_1}^{-2}=
\fg^{-2}$ and ${\fg'_2}^{-2}=0$. It follows in particular that:
\begin{equation} \label{inclusion}
  \fg^2 \oplus [\fg^{-2}, \fg^2] \oplus \fg^{-2} \subset {\fg'_1}_{\C} \;\;.
\end{equation}
As $\fg'_2$ centralizes $\fg'_1$, it centralizes, and hence in particular
normalizes, $\fg^2 \oplus [\fg^{-2}, \fg^2] \oplus \fg^{-2} $. The
normaliser $\fn$ of $\fg^2 \oplus [\fg^{-2}, \fg^2] \oplus \fg^{-2}$
in $\fg$ is therefore a Hodge-Lie subalgebra containing $\fg'_2$. As $\fg^2
\oplus [\fg^{-2}, \fg^2] \oplus \fg^{-2} $ is a $\fg^0$-module, its
normaliser $\fn$ is a $\fg^0$-module too. Thus $\fn^{-1}$ decomposes into a sum of
irreducible $\fg^0$-modules. As already explained in \Cref{firstsectionontyp},
for any non-trivial irreducible $\fg^0$-submodule $E \subset \fg^{-1}$, 
we have that $[E, \fg^2]$ is a non-zero subspace of $\fg^{1}$. In particular
such an $E$ cannot normalize $\fg^2
\oplus [\fg^{-2}, \fg^2] \oplus \fg^{-2}$. 

Therefore we have proven that
$\fn^{-1}=0$, which implies
${\fg'_{2}}^{-1}=0$. Finally $\fg'_{2} = {\fg'_{2}}^0$ is a compact Lie
algebra, establishing the contradiction we were aiming for.

\end{proof}

\section{Proof of \Cref{corol}} \label{proof corol}

\begin{proof}[\unskip\nopunct]
Let $\VV$ be, as in \Cref{corol}, of level at least 3. Using \Cref{reduction} we can without
loss of generality (in particular without changing the level) assume that $\HH= \G^\der$. It then follows from
\Cref{level criterion} (proven as \Cref{level criterion 1}) that $\HL(S, \VV)= \HL(S,
\VV)_{\atyp}$. Let $Z\subset S$ be an irreducible component of the
Zariski-closure of $\HL(S, \VV)_\fpos= \HL(S,
\VV)_{\fpos, \atyp}$. Let us apply \Cref{geometricZP} to $Z$. If we
are in case~(b) of \Cref{geometricZP}, then necessarily $Z= S$, as $S$ does not admit any
{\em strict} typical special subvariety of positive dimension. With the notations of the proof of
\Cref{geometricZP}, we thus obtain a decomposition $\G^\ad= \MM \times \LL$, where $\LL$ is such that $D_L$ is not a
point, and $\MM$ is the algebraic monodromy group of special
subvarieties $Y_i$, $i\in \NN$, contained in $\HL(S, \VV)_\fpos$. But
then the projection of $\Phi({Y_i}^\an)$ on $\Gamma_\LL \backslash
D_L$ is a point: contradiction to the fact that $Y_i$ is factorwise of
positive dimension. Thus we are necessarily in case~(a) of 
\Cref{geometricZP}: $Z$ is a maximal strict (atypical) special
subvariety of $S$ for $\VV$. Thus $\HL(S, \VV)_\fpos$ is algebraic. Finally, If $\G^{\ad}$ is simple, then  
\begin{displaymath}
\HL(S, \VV^\otimes)_\fpos=\HL(S, \VV^\otimes)_\pos=\HL(S, \VV^\otimes)_{\pos, \atyp}
\end{displaymath}
 is algebraic in $S$. \Cref{corol} is therefore proven.
\end{proof}

The proof of \Cref{level2} is essentially the same, replacing \Cref{level criterion} by \Cref{level criterion 12}.
\begin{proof}[Proof of \Cref{level2}]
Assume that $\VV$ is of level $2$ and that $\G^{\ad}$ is simple. 
Thanks to \Cref{level criterion 12}, 
and we obtain that $\HL(S, \VV^\otimes)_{\pos, \atyp}$ is algebraic, as claimed in \Cref{level2}.
\end{proof}

\section{Applications in higher level}\label{newzpandtyp}
In this section we combine the Geometric Zilber--Pink conjecture
(\Cref{geometricZP}) with the fact that, in level $>1$, the typical
Hodge locus is constrained (as established in
\Cref{criterionsection}).  
In particular we prove the results announced in
\Cref{mainsectionintro} and then the applications presented in
\Cref{complapp}.

\subsection{Proof of \Cref{hypersurface} and \Cref{hypersurface2}}
 Let $n \geq 2$ and $d\geq 3$. Let $\PP^{N(n, d)}_\CC$ be the projective space parametrising the hypersurfaces $X$ of
$\PP^{n+1}_\CC$ of degree $d$, with 
\begin{displaymath}
N(n, d)=\binom{n+d+1}{d}-1.
\end{displaymath}

 Let $U_{n, d} \subset \PP^{N(n, d)}_\CC$ be
the Zariski-open subset parametrising the smooth
hypersurfaces $X$ (its complement, the so called \emph{discriminant locus}, is irreducible and of codimension one). 
Let $\G_{n,d}$ be the group of automorphisms of $H^n(X,
  \QQ)_\prim$ preserving the cup-product. When $n$ is odd the primitive 
  cohomology is the same as the cohomology. When $n$ is even
  it is the orthogonal complement of $h^{n/2}$, where $h$ is some fixed
  hyperplane class. Thus $\G_{n,d}$ is either a symplectic or an orthogonal
  group depending on the parity of $n$, and it is a simple
  $\QQ$-algebraic group. 
 One knows that the
  monodromy group $\HH$ of $\VV$ coincides with the simple group
  $\G_{n,d}$, see indeed the remark below. 
  As $\HH \subset \G^\ad \subset
  \G_{n,d}$ we deduce $\G^\ad =\G_{n,d}$, hence $\G^\ad$ is simple.
  
If $n=3$ and $d \geq 5$; $n=4$ and $d\geq 6$; $n=5,6,8$ and $d \geq
4$; and $n=7$ or $\geq 9$ and $d \geq 3$, one checks that the level of
$\G^\ad = \G_{n, d}$ is at least $3$, giving \Cref{hypersurface}. This
follows from  Griffiths' residue theory \cite{zbMATH03341218}, and the
computation of the Hodge diamond of $H^n(X,\Z)_{\prim}$. See in particular \cite[Lecture 4]{zbMATH00719227}, and also
\cite{zbMATH04042054}, for similar computations.  
For example the level is $n$, as soon as $d\geq n+2$, and the level is
at least three also if $n=3$, and $d=5$. If $n=2$, one sees that the
level is $2$, as soon as $d\geq 5$ (it is a well known fact that the
level is one if $d=4$).

The results then follow from \Cref{corol} and \Cref{level2}, proved in \Cref{proof corol}.

\begin{rmk}
Beauville \cite[Theorem 2 and Theorem 4]{Beau}, building on the
  work of Ebeling and Janssen, computes exactly the image of the
  monodromy representation mentioned above. In our argument we just
  need to know $\HH$, the Zariski closure of the image of such a
  monodromy representation. The easier fact that $\HH=\G_{n,d}$
  follows from the \emph{Picard-Lefschetz formulas} and it is due to
  Deligne \cite[Proposition 5.3 and Theorem 5.4]{zbMATH03450362}, see
  also \cite[Section 4.4]{zbMATH03713855}.
\end{rmk}

\subsection{Proof of \Cref{shimuralocus}}\label{proofcor2} 
Let $\VV$ be a polarizable $\ZZ$VHS on a smooth connected complex quasi-projective variety
$S$, and $\Gamma \backslash D$ be the associated Mumford--Tate
domain. We first argue assuming that the adjoint group of the generic
Mumford--Tate $\G$ of $\VV$ is simple. 

Assume that $S$ contains a Zariski-dense set of special subvarieties
$Z_n$ which are of Shimura type (but not necessarly with dominant
period map), and of period dimension $>0$. Heading for a contradiction
let us assume that they are atypical
intersections. \Cref{geometricZP} then implies that $S$ is an
atypical intersection, that is that 
\begin{displaymath}
\codim_{\Gamma \backslash D} (\Phi(S^{\an})) < \codim_{\Gamma
  \backslash D} (\Phi(S^{\an}) ) + \codim_{\Gamma \backslash D}
(\Gamma \backslash D) . 
\end{displaymath}
Which is a contradiction. That means that the $Z_n$ are typical
subvarieties for $(S,\VV)$. As explained in \Cref{firstsectionontyp}, to have a typical intersection between
$\Phi(S^{\an})$ and the special closure of $Z_n$, which we denote by
$\Gamma_{\HH_n} \backslash D_{\HH_n}$ we must have the following equality
at the level of holomorphic tangent spaces (at some point $P \in
Z_n$): 
\begin{equation}\label{eqshimuralocus}
T(S) + T(\Gamma_{\HH_n} \backslash D_{\HH_n}) = T (\Gamma \backslash D).
\end{equation}
By definition, however, the left hand side lies in the horizontal
tangent bundle $T_h(\Gamma \backslash D)$, which was introduced at the
beginning of \Cref{newzpsection}. It follows that $\Gamma \backslash
D$ must be a Shimura variety, as explained in (b) of \Cref{rmk1zp}, since we must have 
\begin{displaymath}
T(D)=T_h(D).
\end{displaymath}
Moreover we can observe that $S$ is Hodge generic, and its Hodge
locus is Zariski-dense. \Cref{shimuralocus} is thereby proven.

For the case where $\G^{\ad}$ is not assumed to be simple we argue as
follows, using the full power of the geometric Zilber--Pink. Assume,
as above, that $S$ contains a Zariski-dense set of special
subvarieties $Z_n$ which are of Shimura type, of period dimension
$>0$, and atypical intersections. Then, by \Cref{geometricZP}, we have
a decomposition $\G^{\ad}=\G'\times \G''$, in such a way that the
$Z_n$ (or a covering thereof) have trivial period dimension for
$\Phi''$. Moreover we can assume that there are no further
decompositions of $(\G', D_{G'})$ in such a way that the $Z_n$ have
trivial period dimension for one of such factors. But then $S$ is
atypical for $\Phi'$, which is a contradiction. That means that the
$Z_n$ have to be typical intersection for $\Phi'$, and so, as
explained previously, $\Gamma ' \backslash D'$ is a Shimura variety.

\begin{rmk}[Shimura varieties inside a period domain]\label{rmkkkk}
Let $D$ be the classifying space for Hodge structures of weight two on
a fixed integral lattice, and $\Gamma \backslash D$ be the
associated space of isomorphism classes. To fix the notation, as in \Cref{firsrmkweight2}, let  
\begin{displaymath}
h^{2,0}-h^{1,1}-h^{0,2}
\end{displaymath}
be the Hodge diamond we are parametrising. To be more explicit we
notice here that $\mathbf{G}=\mathbf{SO}(2h^{2,0},h^{1,1})$. 

Assume that $h^{1,1}$ is even, and write:
\begin{displaymath}
h^{2,0}=p>1,\ \ \ \ h^{1,1}=2q>0.
\end{displaymath} 
Shimura-Hodge subvarieties of $D$ appear for example from the trivial map \emph{from complex to real}
\begin{displaymath}
\SU(p,q) \to \SO(2p,2q),
\end{displaymath}
which indeed induces a geodesic immersion (from the Hermitian space
associated to $\SU(p,q)$ to $D$), as one can read in
\cite[Construction 1.4]{carlson} and \cite{Robles}. 
\end{rmk}

\subsection{Proof of \Cref{modularlocus} -- Special correspondences are atypical intersections}\label{proofcor1} 
Let $\VV$ be a polarizable $\ZZ$VHS on a smooth connected complex quasi-projective variety
$S$. Here we describe the geometry of a portion of the Hodge locus of
$(S\times S, \VV \times \VV)$, which we refer to as the \emph{modular
  locus}, as introduced in \Cref{modularlocusss}.  

Let $\Phi: S^{\an}\to \Gamma \backslash D$ be the period map
associated to $\VV$, and, as usual assume that
$(\mathbf{G},D)=(\mathbf{G}_S,D_S)$, that is that $\Phi(S^{\an})$ is
Hodge generic in $\Gamma \backslash D$. A special correspondence is a
$\dim S $-subvariety of $S\times S$ which come from sub-Hodge datum of
$(\mathbf{G}\times \mathbf{G},D\times D)$ of the form $T_g\subset
\Gamma \backslash D\times \Gamma \backslash D$, for some $g$ in the
commensurator of $\Gamma$. That is  
\begin{displaymath}
(\Phi\times \Phi) (W^{\an})=\left((\Phi\times \Phi) (S^{\an}\times S^{\an})\cap T_g\right)^0.
\end{displaymath}
 To prove \Cref{modularlocus}, via \Cref{geometricZP}, we show that
 the above is an atypical (maximal\footnote{Maximality follows
   directly form the fact that $T_g$ is not contained in any strict
   period sub-domains of $\Gamma \backslash D\times \Gamma \backslash
   D$}) intersection in $\Gamma \backslash D \times \Gamma \backslash
 D$, unless $\Phi$ is dominant, which implies that $\Gamma \backslash
 D$ is a Shimura variety (see for example \cite[Lemma
 4.11]{klin}). Notice that 
\begin{displaymath}
2 \dim \Gamma \backslash D - \dim \Phi (S^{\an})= \codim \Phi (W^{\an}) 
\end{displaymath}
is strictly smaller than
\begin{displaymath}
\codim ((\Phi\times \Phi) (S^{\an}\times S^{\an})) + \codim T_g= 2(
\dim \Gamma \backslash D -\dim \Phi (S^{\an})) + \dim \Gamma
\backslash D, 
\end{displaymath}
(where all codimensions are computed in $\Gamma \backslash D \times \Gamma \backslash D$), if and only if
\begin{displaymath}
\dim \Phi (S^{\an})<\dim  \Gamma \backslash D.
\end{displaymath}
That is, if and only if $\Phi$ is not dominant. \Cref{modularlocus} is therefore a simple consequence of \Cref{geometricZP}.

\section{Typical Locus -- All or nothing}\label{typicalsection}
In this section we prove \Cref{typicallocus}, which, for the reader's convenience, is recalled below. It immediately 
implies \Cref{KO1}, providing the elucidation of \Cref{KO} we were looking for.
\begin{theor}\label{ty1}
Let $\VV$ be a polarizable $\ZZ$VHS
on a smooth connected complex quasi-projective variety 
$S$. If the typical Hodge locus $\HL(S,\VV^\otimes)_{\typ} $ is non-empty then
$\HL(S,\VV^\otimes)$ is analytically (hence Zariski) dense in
$S$.
\end{theor}

As usual, we let
\begin{displaymath}
 \Phi: S^{\an} \to \Gamma \backslash D  
\end{displaymath}
be the period map associated to $\VV$, where
$(\mathbf{G},D)=(\mathbf{G}_S,D_S)$ is the generic Hodge datum associated to
$(S,\VV)$. The proof of the above result builds on a local computation 
(at some smooth point).

\subsection{Proof of \Cref{ty1}}\label{sectionty1}
Let $\Gamma ' \backslash D'$ be a period subdomain of $\Gamma \backslash D$ such that
\begin{equation}\label{typicaleq}
0 \leq \dim ( (\Phi(S^{\an}) \cap \Gamma ' \backslash D')^0) = \dim
\Phi(S^{\an}) + \dim \Gamma ' \backslash D'- \dim \Gamma  \backslash
D, 
\end{equation}
that is we have one typical intersection $Z=\Phi^{-1}(\Gamma '
\backslash D')^0$ (accordingly to \Cref{atypical}). By definition we also have that $Z$ is not singular for $\VV$.
 Let $\mathbf{H}\subset \mathbf{G}$ be the generic Mumford--Tate of $D'$,
and, as usual, write $H=\mathbf{H}(\R)^+ \subset G =\mathbf{G}(\R)^+$ and $\mathfrak{h}$ for its Lie algebra,
which is an $\R$-Hodge substructure of $\mathfrak{g}=
\operatorname{Lie}(G)$. Finally we denote by $N_G(H)$ the normaliser
of $H$ in $G$. Let  
\begin{displaymath}
\mathcal{C}_H:=G/N_G(H)=\{H'= g Hg^{-1}: g \in G\}
\end{displaymath}
be the set of all subgroups of $G$ that are conjugated to $H$
(under $G$) with its natural structure of real-analytic manifold (a
manifold which, unsurprisingly, appeared also in the proof of
\Cref{geometricZP}, see indeed \Cref{mainpropwithas}). Set  
\begin{displaymath}
\Pi_H:=\{(x,H')\in D \times \mathcal{C}_H: x(\DT)\subset H' \}\subset D \times \mathcal{C}_H,
\end{displaymath}
and let $\pi_1$ (resp. $\pi_2$) be the natural projection to $D$
(resp. to $\mathcal{C}_H$). Notice that $\pi_i$ are real-analytic
$G$-equivariant maps (where $G$ acts diagonally on $\Pi_H$). 

Let $\bar{S}$ be the preimage of $\Phi(S^{\an})$ in $D$, along the
natural projection map $D \to \Gamma \backslash D$, and $\tilde{S}$ be the
preimage of $\overline{S}$ in $\Pi_H$, along $\pi_1$. By restricting
$\pi_2$ we have a real-analytic map 
\begin{displaymath}
f :  \tilde{S} \to \mathcal{C}_H.
\end{displaymath}
By a simple topological argument, as explained for example in
\cite[Proposition 1]{chai}, to prove that $\HL(S,\VV^\otimes)$ is
dense in $S$, it is enough to prove that $f$ is generically a submersion
(that is a submersion outside a nowhere-dense real analytic subset $B$
of $\tilde{S}$).\footnote{Because of its simplicity, for the
  convenience of the reader, we recall here Chai's argument. Let
  $\Omega \subset S^\an$ be an open subset. Since $f$ is open, the
  image of $\pi^{-1}(\Omega)-B$ in $\mathcal{C}_H$ is open, hence
  meets in a dense set the $\mathbf{G}(\Q)_+$-conjugates of $H$, as
  $\mathbf{G}(\Q)_+$ is topologically dense in $G$.} As being submersive
is an open condition (for the real analytic topology), it is enough to
find a smooth point in $\tilde{S}$ at which $f$ is submersive. 

Let us analyse this condition. Let $y = (P,H') \in D \times \mathcal{C}_H$ be a point of $\tilde{S}$.
The real tangent space of $\mathcal{C}_H$ at $f(y)$ is canonically isomorphic to
\begin{displaymath}
\mathfrak{g}/N_G(\operatorname{Lie}(H'))=\mathfrak{g}/\mathfrak{n}_{\mathfrak{g}}(\mathfrak{h}').
\end{displaymath}

The image of $df$ at $y=(\overline{P},H')$ is equal to
\begin{equation}\label{eqequality}
(\mathfrak{m}+T_{P}(\overline{S})_{\R}+\mathfrak{n}_{\mathfrak{g}}(\mathfrak{h}'))/\mathfrak{n}_{\mathfrak{g}}(\mathfrak{h}')
\end{equation}
where $\mathfrak{m}$ is the Lie algebra of $M$ (from the
identification $D=G/M$). Thus $f$ is a submersion at $y=(P,H')$ if and
only if (\ref{eqequality}) is equal to
$\mathfrak{g}/\mathfrak{n}_{\mathfrak{g}}(\mathfrak{h}')$.
Now we work on the \emph{complex} tangent
spaces (by tensoring all our $\R$-Lie algebras with $\C$), and use the
fact that $\mathfrak{h}, \mathfrak{n}_{\mathfrak{g}}(\mathfrak{h}'),$ and $\mathfrak{m}$ are naturally endowed with an $\R$-Hodge
structure. Thus
\begin{equation*}\label{eqequality2}
\mathfrak{g}_{\C}= \mathfrak{g}^{-w} \oplus \dots \oplus
\mathfrak{g}^{-1} \oplus \mathfrak{g} ^{0}\oplus
\mathfrak{g} ^{1} \oplus \dots  \oplus
\mathfrak{g}^{w}, 
\end{equation*}
for some $ w\geq 1$. To ease the notation set
\begin{displaymath}
U:= T_{P}(\overline{S})_\C \subset \mathfrak{g}^{-1}.
\end{displaymath} 
It follows that $f$ is a submersion at $y=(P,H')$ if and
only if the following two
 conditions are satisfied (compare with \cite[Proposition
 2]{chai}):
\begin{equation}\label{eq11typ}
 \mathfrak{n}_{\mathfrak{g}}(\mathfrak{h}')^{-k} = \mathfrak{g}^{-k}
\end{equation}
for any $k=w, \dots, 2$ (recall that $\mathfrak{m}$ is pure of type $(0,0)$, so $\mathfrak{m}^{-k}=0$ in this range), and
\begin{equation}\label{eq22typ}
U+\mathfrak{n}_{\mathfrak{g}}(\mathfrak{h}')^{-1}=\mathfrak{g}^{-1}.
\end{equation}

Let us now exhibit a point $y = (P, H') \in \tilde{S}$ satisfying both
(\ref{eq11typ}) and (\ref{eq22typ}). Choose $P$ one lift in $D$ of a
smooth Hodge-generic point of $\Phi(Z^{\an})$ and $H'= \HH'(\RR)^+$ where $\HH'$ is
the Mumford-Tate group at that point. Then (\ref{typicaleq}) implies the following
decomposition of the holomorphic tangent bundle of $D$ at
$P$:
\begin{equation}\label{123}
U + T_{P}(D')_\C=T_{P}(D)_\C=\mathfrak{g}
^{-w} \oplus \dots \oplus \mathfrak{g} ^{-1}.
\end{equation}
In particular we have $(\mathfrak{h}')^{-k}=\mathfrak{g}^{-k}$, for all $k>1$, which implies \eqref{eq11typ} (since $(\mathfrak{h}')^{-k} \subset \mathfrak{h}'$) and $(\mathfrak{h}')^{-1} + U=\mathfrak{g}^{-1}$, which implies that \eqref{eq22typ} is satisfied at $y$.

\Cref{ty1} is proven. 

\begin{rmk}
If $\VV$ is of level two and $\mathbf{G}^{\ad}$ is simple, then the density established above has to
come either from the typical Hodge locus, or from the atypical Hodge
locus of zero period dimension (as $\HL(S, \VV^\otimes)_{\pos, \atyp}$
is algebraic, as established in \Cref{level2}). That is, if
  $\HL(S,\VV^\otimes)_{\pos, \typ} $ is non-empty, then it is dense in
  $S$. 
  \end{rmk}



\section{On a question of Serre and Gross}\label{examplesection} 
To conclude the paper we are left to prove the results announced in \Cref{Picard} and \Cref{The curve}.
\subsection{Proof of \Cref{typicaldiv}, after Chai}\label{sectiongeneral}
We first recall the main theorem of Chai \cite{chai}, whose proof was of inspiration for the argument of \Cref{typicalsection}.

Let $(\mathbf{G},X)$ be a connected Shimura datum, and $\Gamma$ a
torsion free finite index subgroup of $\mathbf{G}(\Z)$. The quotient
$\Gamma \backslash X$ is a smooth quasi-projective variety
\cite{MR0216035}, let $n$ be its complex dimension, which we assume
to be $>0$. Recall that a \emph{special} subvariety of $\Gamma \backslash X$ is, by definition, a component of
the Hodge locus $\HL(\Gamma \backslash X, \VV^\otimes)$, where $\VV$
denotes any $\ZZ$VHS on $\Gamma \backslash X$ defined by a faithful algebraic
representation of $\G$ (this Hodge locus is independent of the choice
of such a representation).

Let $(\mathbf{H},X_H)$ be a sub-Shimura datum of $(\mathbf{G},X)$, and 
let $S \subset \Gamma \backslash X$ be an irreducible closed subvariety. 
As usual we set $H=\mathbf{H}(\R)^+ \subset G= \G(\R)^+$.
Consider the following subset of 
$\HL(S, \VV^\otimes)$:
\begin{equation}
\HL(S,\mathbf{H}):=\{x \in S: \MT(x) \subset g\mathbf{H}g^{-1}\text{  for some  } g\in \mathbf{G}(\Q)_+ \}.
\end{equation}

\begin{theor}[{\cite{chai}}]\label{chairecall}
There exists a constant $c=c(G,X,H) \in \N$ such that if $S$ has codimension at most $c$ in $\Gamma \backslash X$, 
then $\HL(S,\mathbf{H})$ is (analytically) dense in $S$.
\end{theor}

To establish \Cref{typicaldiv} we just need to prove the following, which builds on an explicit computation of 
$c=c(G,X,H)$ given by Chai:
\begin{lem}\label{lemmachai}
Assume that $\mathbf{G}$ is absolutely simple, $\Gamma_H \backslash X_H \subset \Gamma \backslash X$ has dimension one, and $N_G(H)=H$.
Then $c(G,X,H)>0$. 
\end{lem}
\begin{proof}
Let $h: \DT \to H$ be a Hodge generic point of $X_H$. Let $K_h$ be the centraliser of $h$ in $G$.
And recall that, as explained in \Cref{firstsectionontyp}, $\fg^{-1}$ is an irreducible $K_h\otimes \C$-module.

Chai \cite[beginning of page 409]{chai} proves\footnote{That is to say that, in the notation of Definition 2.2 in \emph{op. cit.},
$c(G,h,H)=d(K_h, \fg^{-1},\mathfrak{h}^{-1})$.}
 that $c(G,h,H)$ is the largest non-negative integer 
such that for every $\C$-vector subspace $W$ of $\fg^{-1}$ of codimension at most $c(G,h,H)$,
there exists an element $k \in K_h$ with
\begin{displaymath}
\fg^{-1}= W+ \Ad k (\mathfrak{h}^{-1})\;\;.
\end{displaymath}
Notice that the sub-$\C$-vector space of $\fg^{-1}$ generated by
$\bigcup_{k \in K_h} \Ad k(\mathfrak{h}^{-1})$
is stable under the action of $K_h$, and therefore, by the irreducibility of $\fg^{-1}$ as $K_h$-module, it is equal to $\fg^{-1}$. 
As $\mathfrak{h}^{-1} $ is one dimensional, there must exist a $k_0 \in K_h$ for which 
\begin{displaymath}
\Ad k_0 (\mathfrak{h}^{-1}) \nsubseteq W.
\end{displaymath}
That is, if $W$ has codimension one in $\fg^{-1}$, then $\fg^{-1}=
W+\Ad k_0(\mathfrak{h}^{-1})$.
Which is saying that $c(G,h,H) \geq 1$, as desired.
\end{proof}

Combining the above lemma and \Cref{chairecall} we obtain the following, which implies \Cref{typicaldiv}:
\begin{theor}[Chai + $\epsilon$]\label{thmfinalechai}
Assume that $\mathbf{G}$ is absolutely simple, $X_H \subset X$ has dimension one, and that $N_G(H)=H$.
If $S$ has codimension one in $\Gamma \backslash X$, then 
$\HL(S,\mathbf{H})$ is dense in $S$. That is 
\begin{equation}\label{inteq}
\bigcup_{g\in \mathbf{G}(\Q)_+} S \cap T_g (\Gamma_H \backslash X_H)
\end{equation}
is dense in $S$. \end{theor}

\begin{rmk}
Under the Andr\'{e}--Oort conjecture for Shimura varieties,
\Cref{thmfinalechai} could be upgraded saying that the \emph{typical}
Hodge locus is in fact dense in $S$ (as we will do in the next section in a special case),
as long as $S$ is not a special subvariety of $\Gamma \backslash X$. 
That is to say that the intersections appearing in (\ref{inteq}) are, up to a finite number,
 zero-dimensional and with Mumford-Tate group equal to $g \mathbf{H}g^{-1}$, 
 for some $g\in \mathbf{G}(\Q)_+$.
\end{rmk}

 \subsection{Proof of \Cref{serrequestion} -- Typical and atypical intersections altogether}\label{finalsection}
Let $\mathcal{M}_g$ be the moduli space of curves of genus $g$, and
\begin{displaymath}
j : \mathcal{M}_g \hookrightarrow \mathcal{A}_g
\end{displaymath}
be the Torelli morphism. Recall that the left hand side is irreducible
and has dimension $3g-3$, while the right hand side has dimension $\frac{g(g+1)}{2}$. 
We denote the image of $j$ by
\begin{displaymath}
\mathcal{T}_g^0=j (\mathcal{M}_g)\subset \mathcal{A}_g,
\end{displaymath}
which will be referred to as the \emph{open Torelli locus}, and by
$\mathcal{T}_g$ its Zariski closure (the so called \emph{Torelli
  locus}). To be more precise a $n$-level structure, for some $n>5$, should be fixed. 
  If $g>3$, it is well
known that $\mathcal{T}_g$ is Hodge generic in $\mathcal{A}_g$. 
This can indeed be observed by using the fact that $\mathcal{T}_g$ 
is not special and that the fundamental group of $\mathcal{M}_g$ surjects 
onto the one of $ \mathcal{A}_g$, see for example \cite[Remark 4.5]{moonenoort}.

From now on, we set $g=4$. This is important because it is the only $g$ for which 
\begin{displaymath}
\codim_{\mathcal{A}_g} (\mathcal{T}_g)=1=10-9= \codim_{\mathcal{A}_4} (\mathcal{T}_4).
\end{displaymath}
Recall that $\mathcal{A}_4$ contains a special curve $Y$ whose generic
Mumford-Tate group $\mathbf{H}$ is isogenous to a $\Q$-form of $\Gm \times
(\Sl_2)^3/\C$, as proven by Mumford \cite{zbMATH03271259}. Mumford's construction is 
such that the normaliser of $\mathbf{H}$
in $\mathbf{G}$ is indeed $\mathbf{H}$. \Cref{thmfinalechai}
shows that $\HL(\mathcal{T}_4, \mathbf{H})$ is analytically dense\footnote{What makes \Cref{serrequestion}
  interesting is that $Y$ is a special subvariety of $\cA_4$ which is
  not of PEL type (while every special subvariety of $\cA_2$ and
  $\cA_3$ is necessary of PEL type).} in $\mathcal{T}_4$. This means that $\mathcal{T}_4$ cuts many Hecke
translates $Y_n$ of $Y$. In particular, upon extracting a sequence of $Y_n$, we have
\begin{equation}
Y_n \cap \mathcal{T}_4^0 \neq \emptyset,
\end{equation}
since it is not possible that all intersections happen on $\mathcal{T}_4 - \mathcal{T}_4^0 $.

Since the $Y_n$ are one dimensional, for each $P \in Y_n $, there are only two possibilities:
\begin{itemize}
\item $P$ is a special point, i.e. $\MT(P)$ is a torus (in which case
  it corresponds to a principally polarised abelian 4-fold with CM);  
\item $\MT(P)= \MT(Y_n)$, which is isogenous to some $\Q$-form of $\Gm \times (\Sl_2)^3$.
\end{itemize}
Therefore, to conclude, we have to find a $n$ and a non-special $P
\in Y_n \cap \mathcal{T}_4^0$. Heading for a contradiction, suppose
that, for all $n$, all points of $Y_n \cap \mathcal{T}_4^0$ are
special. By density of $\HL( \mathcal{T}_4, \mathbf{H})$ in $\mathcal{T}_4$, 
this means that $\mathcal{T}_4$ contains a dense
set of special points. Andr\'{e}--Oort now implies that
$\mathcal{T}_4$ is special, which is the contradiction we were looking 
for. This shows the existence of a point of $\mathcal{T}_4^0(\C)$,
corresponding to a curve $C/\C$ of genus $4$ whose Jacobian has
Mumford-Tate group isogenous to a $\Q$-form of the $\C$-group $\Gm \times
(\Sl_2)^3$. To conclude the proof of \Cref{serrequestion} we just
have to observe that $\mathcal{M}_4, \mathcal{A}_4, j$ and all the
$Y_n$ can be defined over $\Qbar\subset \C$ (see for example
\cite{MR791585} or think about the moduli problem they are
describing), and therefore all intersections we considered during the
proof are defined over $\Qbar$. The result follows.

\begin{rmk}\label{rmkao}
The Andr\'{e}--Oort conjecture of $\cA_g$ is a deep theorem whose
proof appears in \cite[Theorem 1.3]{MR3744855}, and builds on the work
of several people. We refer to \cite{MR3821177} for the history of the
proof. We note here that, for $1< g\leq 6$, it is known by combining
the Pila-Zannier strategy with the Ax--Lindemann conjecture and the
Galois bound appearing in \cite{zbMATH06195219} (which is
 easier than the Galois bound needed in $\cA_g$, for $g$
arbitrary). This is explained in \cite[Theorem 5.1]{MR3177266} and in
\cite[Theorem 1.3]{zbMATH06284346}. 
\end{rmk}

\begin{rmk}[Oort's question]\label{oortrmk}
In 
\cite[Section 7]{zbMATH07128855} it is asked whether one of Mumford's
special curves $Y\subset \mathcal{A}_4$ can actually lie in
$\mathcal{T}_4^0$. As a simple application of a result of Toledo
\cite{zbMATH04057672} we show that there are no such $Y$s. In
\emph{op. cit.} it is proven that if a compact geodesic curve
$Z\subset \mathcal{A}_g$ is contained in $\mathcal{T}_4^0$ then it has
curvature $1/l$ with $1< l \leq (g-1)/3$. If $g= 4$ it means that $l$,
which is a priori an integer in the interval $(1,g)$, is both $\leq 1$
and $>1$. Since Mumford's curves are compact, this shows the desired
claim. 
\end{rmk}

\bibliographystyle{abbrv}

\bibliography{biblio.bib}

\Addresses

\end{document}